\documentclass[12 pt,reqno]{amsart}
\usepackage{enumitem}
\usepackage{amsmath,amsthm,amssymb,amscd}
\usepackage{upgreek}
\usepackage{cite}
\usepackage{url}
\usepackage{mathtools}
\usepackage{fullpage}
\usepackage{float}
\usepackage{comment}
\usepackage{xcolor}

\usepackage[T1]{fontenc}
\usepackage{lmodern}

\DeclareFontFamily{U}{rcjhbltx}{}
\DeclareFontShape{U}{rcjhbltx}{m}{n}{<->rcjhbltx}{}
\DeclareSymbolFont{hebrewletters}{U}{rcjhbltx}{m}{n}

\DeclareMathSymbol{\mem}{\mathord}{hebrewletters}{109}

\usepackage{etoolbox}
\AtBeginEnvironment{tabular}{\onehalfspacing}

\input xy
\xyoption{all}

\usepackage{microtype}

\allowdisplaybreaks[1]

\usepackage[bookmarksopen,bookmarksdepth=3]{hyperref}

\newtheorem{thm}{Theorem}
\newtheorem{prop}{Proposition}[section]
\newtheorem{lm}[prop]{Lemma}

\theoremstyle{definition}
\newtheorem{dfn}[prop]{Definition}

\theoremstyle{remark}
\newtheorem{remarkx}[prop]{Remark}

\newenvironment{rem}
  {\pushQED{\qed}\remarkx}
  {\popQED\endremarkx}

\DeclareMathOperator{\rk}{rk}

\DeclareMathOperator{\Id}{Id}

\newcommand{\fix}{\textit{fix}}

\newcommand{\lrarr}{\longrightarrow}
\newcommand{\Rarr}{\Rightarrow}

\newcommand{\R}{\mathbb{R}}

\newcommand{\Z}{\mathbb{Z}}

\newcommand{\M}{\mathcal{M}}
\newcommand{\J}{\mathcal{J}}

\renewcommand{\P}{\mathbb{C}P}
\renewcommand{\L}{\Lambda}
\renewcommand{\l}{\lambda}
\newcommand{\mI}{\mathcal{I}}

\newcommand{\Hh}{\widehat{H}}

\newcommand{\Ah}{\widehat{A}}

\newcommand{\p}{\mathfrak{p}}
\newcommand{\pkl}{\p_{k,l}}

\newcommand{\q}{\mathfrak{q}}

\newcommand{\m}{\mathfrak{m}}
\renewcommand{\d}{\partial}

\newcommand{\at}{\tilde{\alpha}}

\newcommand{\xit}{\tilde{\xi}}
\newcommand{\etat}{\tilde{\eta}}
\newcommand{\zetat}{\tilde{\zeta}}

\newcommand{\Mt}{\widetilde{\M}}

\newcommand{\bt}{\tilde b}
\newcommand{\gt}{\tilde{\gamma}}
\newcommand{\mt}{\tilde\m}
\newcommand{\qt}{\tilde\q}
\newcommand{\ct}{\tilde{c}}
\newcommand{\ot}{\tilde{o}}
\newcommand{\mC}{\mathfrak{C}}
\newcommand{\mD}{\mathfrak{D}}
\newcommand{\evbt}{\widetilde{evb}}
\newcommand{\evit}{\widetilde{evi}}
\newcommand{\evt}{\widetilde{ev}}

\renewcommand{\u}{\upsilon}

\newcommand{\mg}{\m^{\gamma}}

\newcommand{\mgt}{\mt^{\gt}}

\newcommand{\qbgt}{\qt^{\bt,\gt}}

\renewcommand{\ll}{\langle\!\langle}
\renewcommand{\gg}{\rangle\!\rangle}

\renewcommand{\Im}{\mbox{Im}}

\newcommand{\Oh}{\widehat{\Omega}}
\newcommand{\Ob}{\overline{\Omega}}

\newcommand{\RP}{\mathbb{R}P}

\newcommand{\sly}{\Pi}
\newcommand{\pr}{\varpi}
\newcommand{\A}{\mathcal{A}}

\newcommand{\D}{\mathcal{D}}

\newcommand{\sababa}{sababa}

\newcommand{\Ups}{\Upsilon}

\renewcommand{\i}{\mathfrak{i}}

\newcommand{\qbg}{\q^{b,\gamma}}
\newcommand{\qg}{\q^{\gamma}}
\newcommand{\qtbg}{\qt^{\bt,\gt}}

\renewcommand{\a}{\alpha}

\newcommand{\pg}{\p^{\gamma}}
\newcommand{\pt}{\tilde{\p}}

\newcommand{\pgt}{\pt^{\gt}}
\newcommand{\ptbg}{\pt^{\bt,\gt}}
\newcommand{\pbg}{\p^{b,\gamma}}

\newcommand{\mR}{\mathfrak{R}}

\newcommand{\uu}{\mathbf{u}}

\newcommand{\ssly}{{S}}
\newcommand{\s}{\mathfrak{s}}

\newcommand{\Lc}{\Lambda_c}
\newcommand{\mbg}{\m^{b,\gamma}}

\newcommand{\thetat}{\tilde{\theta}}

\newcommand{\mP}{\mathcal{P}}
\newcommand{\mQ}{\mathcal{Q}}

\begin{document}

\title{Open Gromov-Witten theory for cohomologically incompressible Lagrangians}

\author[S. Tukachinsky]{Sara B. Tukachinsky}
\address{School of Mathematical Sciences\\ Tel Aviv University\\Tel Aviv, 6997801, Israel }\email{sarabt1@gmail.com}

\date{November 2023}

\begin{abstract}
This paper classifies separated bounding pairs for Lagrangian submanifolds that are homologically trivial inside the ambient space, under the assumption that restriction on cohomology from the ambient space to the Lagrangian is surjective.
As an application, open Gromov-Witten invariants are defined under the above assumptions.
When the Lagrangian is the fixed locus of an anti-symplectic involution, the surjectivity assumption can be somewhat relaxed while the classifying space needs to be modified.
\end{abstract}

\maketitle

\setcounter{tocdepth}{2}
\tableofcontents

\section{Introduction}

The idea that bounding chains can be used to define genus zero open Gromov-Witten (OGW) invariants was first pointed out by Joyce~\cite{Joyce}. The resulting invariants depend on the gauge equivalence class of the bounding chain used to define them.
In~\cite{ST2}, the idea was introduced that classification of bounding pairs, up to gauge equivalence, allows singling out a canonical choice of (an equivalence class of) a bounding chain. This, accordingly, gives rise to a canonical definition of genus zero OGW invariants.
In~\cite{ST3}, an open WDVV equation was proven for a particular type of bounding chains we called ``separated''. Separatedness is a natural condition in OGW theory, designed to make a clear distinction between boundary and interior constraints.

In the current work, we classify separated bounding pairs, up to gauge equivalence, for cohomologically incompressible Lagrangian submanifolds that are homologically trivial inside the ambient symplectic manifold.
In the real setting, meaning, when the Lagrangian is the fixed locus of an anti-symplectic involution, the incompressibility assumption can be somewhat relaxed, while simultaneously the real part of the classifying space is taken.

As an application, we can choose a canonical gauge-equivalence class of separated bounding pairs, thus defining OGW invariants under the above topological assumptions.
The determination of which class should be the ``correct'' choice hinges on two observations. One is the classification theorem(s) proved in this manuscript. The other is~\cite[Remark 4.17]{ST3}, where we explain that the open superpotential naturally favors boundary constraints of top cohomological degree. More concretely, regardless of the dimension of the affine space classifying all (separated) bounding pairs, at most one dimension will be visible in the resulting OGW invariants -- the one corresponding to point boundary constraints. Thus, the canonical choice for a bounding chain is a point-like one.
The fact that there are dimensions that do not manifest in the resulting invariants can be interpreted as saying that the open Gromov-Witten invariants are truly invariants of $L$, not of an element of the Fukaya category of $X$.

The method of proof is obstruction theory as introduced in~\cite{FOOO}, used analogously to~\cite{ST2}.
To prove the vanishing of obstruction classes, we use an open-closed map. This is reminiscent to -- though different from -- \cite[Theorem 3.8.11]{FOOO}, see Section~\ref{sssec:related} for a discussion.
Our version of the open-closed map uses push-forward of differential forms, and thus takes values in currents.
In the proof of injectivity of the classifying map, Section~\ref{ssec:inj}, we use currents over a dga.
An alternative argument using primarily differential forms is outlined in Section~\ref{sssec:alternate}.

\subsection{Setting}


The terminology used throughout this paper is consistent with that of~\cite{ST1, ST2, ST3}, and much of the exposition in this section is taken from there.

\subsubsection{Basic notation}

Consider a symplectic manifold $(X,\omega)$ of dimension $2n>0$ and a connected, Lagrangian submanifold $L$ with relative spin structure $\s$. Let $J$ be an $\omega$-tame almost complex structure on $X.$ Denote by $\mu:H_2(X,L) \to \Z$ the Maslov index.

Let $\sly$ be a quotient of $H_2(X,L;\Z)$ by a (possibly trivial) subgroup $S_L$ contained in the kernel of the homomorphism $\omega \oplus \mu : H_2(X,L;\Z) \to \R \oplus \Z.$ Thus the homomorphisms $\omega,\mu,$ descend to $\sly.$ Denote by $\beta_0$ the zero element of $\sly.$ Let
\begin{equation}\label{eq:varpi}
\pr: H_2(X;\Z) \to \sly
\end{equation}
denote the composition of the natural map $H_2(X;\Z) \to H_2(X,L;\Z)$ with the projection $H_2(X,L;\Z) \to \sly.$

\subsubsection{Rings}\label{sssec:rings}
Define Novikov coefficient rings
\begin{gather*}
\L=\left\{\sum_{i=0}^\infty a_iT^{\beta_i}\bigg|a_i\in\R,\beta_i\in \sly,\omega(\beta_i)\ge 0,\; \lim_{i\to \infty}\omega(\beta_i)=\infty\right\},\\
\Lc:= \left\{\sum_{j=0}^\infty a_jT^{\varpi(\beta_j)}\,| \,a_j\in \R, \beta_j\in H_2(X;\Z), \omega(\beta_j)\ge 0,\lim_{j\to \infty}\omega(\beta_j)=\infty\right\}\leqslant \L.
\end{gather*}
Gradings on $\L,\Lc$ are defined by declaring $T^\beta$ to be of degree $\mu(\beta).$ Define ideals $\L^+ \triangleleft \L$ and $\Lc^+ \triangleleft \Lc$ by
\[
\L^+=\left\{\sum_{i=0}^\infty a_iT^{\beta_i} \in \L \bigg|\;\omega(\beta_i)> 0 \quad\forall i\right\}, \qquad \Lc^+ = \Lc\cap\L^+ .
\]

Let $s_0,\ldots, s_M,t_0,\ldots,t_N,$ be formal variables with degrees in $\Z$.
Define graded-commutative rings
\[
R:=\L[[s_0,\ldots, s_M,t_0,\ldots,t_N]],\quad Q:=\L_c[[t_0,\ldots,t_N]] ,
\]
thought of as differential graded algebras with trivial differential.

Define a valuation
\[
\nu:R\lrarr \R,
\]
by
\[
\nu\left(\sum_{j=0}^\infty a_jT^{\beta_j}\prod_{i=0}^Nt_i^{l_{ij}}\right)
= \inf_{\substack{j\\a_j\ne 0}} \left(\omega(\beta_j)+\sum_{i=0}^N l_{ij}\right).
\]

Denote by $A^*(L)$ and $A^*(X)$ the rings of differential forms on $L$ and $X$, respectively, with coefficients in $\R$.
Denote by
\[
\Ah^*(X,L):=\left\{\eta\in A^*(X)\;\bigg|\,\int_L\eta=0\right\}
\]
the subcomplex of differential forms on $X$ consisting of those with trivial integral on $L$.
We write $\Hh^*(X,L)$ for its cohomology.
Denote by $\A^k(X)$ the space of currents of cohomological degree $k$, that is, the continuous dual space of $A^{\dim X - k}(X)$ with the weak-$\ast$
 topology.
We identify $A^k(X)$ as a subspace of $\A^k(X)$ via
\begin{gather*}
\varphi:A^k(X)\hookrightarrow \A^k(X),\\
\varphi(\gamma)(\eta)=\int_{X}\gamma\wedge\eta,\quad \eta\in A^{\dim X -k}(X).
\end{gather*}
See Section~\ref{ssec:currconv} for details and conventions.

Completing with respect to $\nu$ the tensor product of differential graded algebras, we take rings of differential forms or currents with coefficients in $R$ or $Q$. Thus,
$A^*(L;R)=A^*(L)\otimes R$, $\Ah^*(X,L;Q)=\Ah^*(X,L)\otimes Q$, and $\A^*(X;R)=\A^*(X)\otimes R$. Similarly for the cohomologies $H^*(L;R)$ and $\Hh^*(X,L;Q)$.
The valuation $\nu$ induces a valuation on the above rings, which we also denote by $\nu$.

Define $\mI_R: = \{\alpha\in R\,|\,\nu(\alpha)>0\},$ and similarly $\mI_Q: = \{\alpha\in Q\,|\,\nu(\alpha)>0\}$.

\subsubsection{Degree notation convention} Given $\alpha$, a homogeneous current (or, in particular, a differential form) with coefficients in $R$, denote by $|\alpha|$ the degree of the current, ignoring the grading of $R$. Denote by $\deg \alpha$ the total grading combining the differential form degree and the grading of $R.$

\subsubsection{Moduli spaces}
Let $\M_{k+1,l}(\beta)$ be the moduli space of genus zero $J$-holomorphic open stable maps $u:(\Sigma,\d \Sigma) \to (X,L)$ of degree $[u_*([\Sigma,\d \Sigma])] = \beta \in \sly$ with one boundary component, $k+1$ boundary marked points, and $l$ interior marked points. The boundary points are labeled according to their cyclic order. The space $\M_{k+1,l}(\beta)$ carries evaluation maps associated to boundary marked points $evb_j^\beta:\M_{k+1,l}(\beta)\to L$, $j=0,\ldots,k$, and evaluation maps associated to interior marked points $evi_j^\beta:\M_{k+1,l}(\beta)\to X$, $j=1,\ldots,l$.

Let $\M_{l+1}(\beta)$ be the moduli space of genus zero $J$-holomorphic stable maps $u: \Sigma \to X$ of degree $u_*([\Sigma]) = \beta\in H_2(X;\Z)$ with $l+1$ marked points. The space $\M_{l+1}(\beta)$ carries evaluations maps $ev_j^\beta:\M_{l+1}(\beta)\to X$, $j = 0,\ldots,l.$

\subsubsection{Regularity assumptions}\label{sssec:regularity}
In line with the trilogy~\cite{ST1,ST2,ST3},
we assume that all $J$-holomorphic genus zero open stable maps with one boundary component are regular, the moduli spaces $\M_{k+1,l}(\beta;J)$ are smooth orbifolds with corners, and the evaluation maps $evb_0^\beta$ are proper submersions. Furthermore, we assume that all the moduli spaces $\M_{l+1}(\beta)$ are smooth orbifolds and $ev_0$ is a submersion.
Examples include $(\P^n,\RP^n)$ with the standard symplectic and complex structures or, more generally, flag varieties, Grassmannians, and products thereof. See Example 1.5 and Remark 1.6 in~\cite{ST1}.
Throughout the paper we fix a connected component $\mathcal{J}$ of the space of $\omega$-tame almost complex structures satisfying our assumptions.  All almost complex structures are taken from $\J.$ The results and arguments of the paper should extend to general target manifolds with arbitrary $\omega$-tame almost complex structures by use of the virtual fundamental class techniques of~\cite{Fukaya,Fukaya2,FOOOtoricI,FOOOtoricII,FOOO1} or the polyfold theory of~\cite{HoferWysockiZehnder,HoferWysockiZehnder1,HoferWysockiZehnder2,HoferWysockiZehnder3,LiWehrheim}.

\subsubsection{\texorpdfstring{$A_\infty$}{A-infty} operations}\label{sssec:intro_m}

For any $\gamma\in \mI_Q \Ah^*(X,L;Q)$ with $d\gamma =0$ and $\deg \gamma=2$, we define operations
\[
\mg_k:A^*(L;R)^{\otimes k}\lrarr A^*(L;R)
\]
by
\begin{multline*}
\mg_k(\alpha_1,\ldots,\alpha_k):=
(-1)^{\sum_{j=1}^kj(|\alpha_j|+1)+1}
\sum_{\substack{\beta\in\sly\\l\ge 0}}\frac{1}{l!} T^{\beta}{evb_0^\beta}_* (\bigwedge_{j=1}^l(evi_j^\beta)^*\gamma\wedge \bigwedge_{j=1}^k (evb_j^\beta)^*\alpha_j)+ \delta_{k,1}\cdot d\alpha_1.
\end{multline*}
The push-forward $(evb_0^\beta)_*$ is defined by integration over the fiber; it is well-defined because $evb_0^\beta$ is a proper submersion. The case $k=0$ is understood as $(evb_0)_*1$.
Intuitively, the $\gamma$ should be thought of as yielding interior constraints, while $\alpha_j$ are boundary constraints. Then the output is a cochain on $L$ that is ``Poincar\'e dual'' to the image of the boundaries of disks that satisfy the given constraints.
The sequence of operations $\{\mg_k\}_{k=0}^\infty$ form an $A_\infty$ structure on $A^*(L;R)$.
More details on these operators are given in Sections~\ref{sssec:m}-\ref{sssec:mt}.

\subsubsection{Closed operations}\label{sssec:intro_clq}

We define similar operations using moduli spaces of closed stable maps,
\[
\qg_{\emptyset,l}:A^*(X;Q)^{\otimes l}\lrarr A^*(X;Q),
\]
as follows.
Let $\gamma \in \mI_Q \Ah^*(X,L;Q)$ with $d\gamma =0$ and $\deg \gamma =2$.
Recall that the relative spin structure $\s$ on $L$ determines a class $w_{\s} \in H^2(X;\Z/2\Z)$ such that $w_2(TL) = i^* w_{\s}$. By abuse of notation we think of $w_{\s}$ as acting on $H_2(X;\Z)$. Set
\begin{equation}\label{eq:qemptyset}
\qg_{\emptyset,l}(\eta_1,\ldots,\eta_l):=
\sum_{\substack{\beta\in H_2(X;\Z)\\m\ge 0}}
(-1)^{w_\s(\beta)}
\frac{1}{m!}T^{\varpi(\beta)}{ev_0^\beta}_* (\bigwedge_{j=1}^l(ev_j^\beta)^*\eta_j\wedge \bigwedge_{j=l+1}^{l+m}(ev_j^\beta)^*\gamma).
\end{equation}
The sign $(-1)^{w_\s(\beta)}$ is designed to balance out the sign of gluing spheres as explained in~\cite[Lemma 2.12]{ST1}.

More details on these operators are given in Sections~\ref{sssec:clq}-\ref{sssec:clqt}.

\subsubsection{Open-closed operations}

We define the open-closed map
\[
\pg_k:A^*(L;R)^{\otimes k} \lrarr \A^*(X;R)
\]
by
\[
\pg_k(\otimes_{j=1}^k\alpha_j)=
(-1)^{\sum_{j=1}^k(n+j)(|\a_j|+1)}\sum_{\substack{\beta\in \sly\\ l\ge 0}} \frac{1}{l!}T^\beta (evi_0^\beta)_*\big(\bigwedge_{j=1}^levi_j^*\gamma\wedge \bigwedge_{j=1}^kevb_j^*\a_j \big).
\]
Note that, since $evi_0$ is not necessarily a submersion, the output is a current. See Section~\ref{ssec:currconv} for our conventions on currents. The operator satisfies properties analogous to those of the $A_\infty$ operators. More details are given in Sections~\ref{sssec:p}-\ref{sssec:pt}.

\subsubsection{Bounding chains}\label{sssec:intro_bdpair}

As in~\cite{ST2}, we define
a \textbf{bounding pair} with respect to $J$ to be a pair $(\gamma,b)$ where $\gamma\in \mathcal{I}_Q D$ is closed with $\deg_D\gamma=2$ and $b\in \mathcal{I}_R C$ with
$
\deg_Cb=1,
$
such that
\begin{equation}\label{eq:bc}
\sum_{k\ge 0}\m_k^\gamma(b^{\otimes k})=c\cdot 1, \qquad c\in \mI_R,\;\deg_Rc=2.
\end{equation}
In this situation, $b$ is called a \textbf{bounding chain} for $\mg$.
We say a bounding chain $b$ is \textbf{separated} if
\[
\int_L b\in \Lambda[[s_0,\ldots,s_M]] \subset R.
\]
Bounding pairs come with a natural equivalence relation called gauge-equivalence.
For more details see Section~\ref{sssec:bdpair}.

\subsection{Results}
\subsubsection{Cohomologically incompressible Lagrangians}\label{sssec:thmA}

Assume $i^*:H^j(X;\R)\to H^j(L;\R)$ is surjective for all $j<n$.
Let $\theta_0,\ldots,\theta_r\in A^*(X)$ closed, homogeneous forms such that $i^*[\theta_j]$ form a basis for $H^{<n}(L;\R)$.
Assume in addition that $[L]=0\in H_n(X;\R)$, and take $\zeta\in \A^{n-1}(X;\R)$ such that $d\zeta = i_*1$. The existence of such $\zeta$ is verified in Lemma~\ref{lm:exact}.

In section~\ref{ssec:xidef} we show that the following map is well defined:

\begin{gather*}
\Xi: \{\text{separated bounding pairs}\}/\sim\quad \lrarr \quad (\mI_Q\Hh^*(X,L;Q))_2\oplus \bigoplus_{m=0}^r (\mI_R)_{1-n+|\theta_m|},\\
\Xi(b,\gamma)=
[\gamma]\oplus \bigoplus_{m=0}^r\langle\pbg_{0,0}-(-1)^n\qg_{\emptyset,1}(\zeta),\theta_m\rangle_X.
\end{gather*}

The range is of the correct degree by direct computation, as is verified in Lemma~\ref{lm:ujdeg}.

\begin{thm}[Classification]\label{thm:A}
The map $\Xi$ is bijective.
\end{thm}

The proof is carried out in Section~\ref{sec:thm1}.
It is based on obstruction theory as introduced by~\cite{FOOO}, and similar to the classification for homological spheres carrie dout in~\cite{ST2}. Conceptually, the contribution of $\pbg$ in $\Xi$ is a generalization of integration over $L$, and the integral of $\qg_{\emptyset}(\zeta)$ is added to balance out cases where the boundary of disks degenerate to a point. The exactness of $L$ is what allows us to remove specifically sphere contributions in order to achieve invariance under gauge equivalence (Lemma~\ref{lm:invt}), rather then eliminating all type-$\D$ summands as was done in~\cite{ST2}. It is used again in the proofs  of Lemmas~\ref{lm:ojex} and~\ref{lm:otjex} in showing that obstruction classes are exact.

A key difference compared to the case covered in \cite{ST2} is in proving the injectivity of the classifying map, specifically in Lemma~\ref{lm:otjex} -- exactness of the obstructions to gauge-equivalence. Here it is essential that on the pseudo-isotopy, all the structure (the maps and the pairing) is defined over the dg algebra $\mR=A^*(I;\R)$.

As detailed in Section~\ref{sssec:alternate} below, instead of using the open-closed map, we could have formulated the results in terms of the operator $\q_{-1,1}$. The proof seems to go through just as well, although I have not verified all the details. Such a phrasing would be reminiscent to the approach of \cite{Hugtenburg1}. But the open-closed map considered as a current seems more illuminating and flexible in that it allows us to work on chains directly, without keeping track of the way they act on other chains.

\begin{rem}
While the range of $\Xi$ depends only on $\rk H^j(L;\R)$, the map itself depends a-priori on the choices of the chains $\zeta$ and $\theta_0,\ldots,\theta_r$. Accordingly, the identification in Theorem~\ref{thm:A} is very much not canonical.

More concretely, the choice of $\zeta$ corresponds to a choice of an $(n+1)$-chain $V$ such that $\d V=L$. Given another chain $V'$ so that $[V]=[V']\in H_{n+1}(X,L;\R)$, it corresponds to $\zeta' \in\A^{n-1}(X;\R)$ such that $\zeta'-\zeta=d\alpha$ for some $\alpha \in\A^{n-2}(X;\R)$. Then,
\[
(\pbg_{0,0}-(-1)^n\qg_{\emptyset,1}(\zeta)) - (\pbg_{0,0}-(-1)^n\qg_{\emptyset,1}(\zeta'))
= (-1)^n\qg_{\emptyset,1}(d\alpha)
= 0.
\]
So, a choice of $\zeta$ corresponds to a choice of a homology class $[V]\in H_{n+1}(X,L)$ such that $\d V=L$.

To analyze the dependence on $\theta_m$, consider $\theta'_m$ with $\theta'_m-\theta_m=d\eta$. Using Lemma~\ref{lm:p0cl} below, we get
\begin{multline*}
\langle\pbg_{0,0}-(-1)^n\qg_{\emptyset,1}(\zeta),\theta'_m\rangle_X - \langle\pbg_{0,0}-(-1)^n\qg_{\emptyset,1}(\zeta),\theta_m\rangle
=\langle\pbg_{0,0}-(-1)^n\qg_{\emptyset,1}(\zeta),d\eta\rangle_X\\ 
= \langle d\pbg_{0,0} - (-1)^nd\qg_{\emptyset,1}(\zeta),\eta\rangle_X
= (-1)^{n+1}\langle \qg_{\emptyset,1}(i_*1),\eta\rangle_X,
= (-1)^{n+1}\i(\qg_{\emptyset,1}(\eta)).
\end{multline*}
So, the choice of $\theta_m$ is significant even within cohomology class.
\end{rem}

\subsubsection{The real spin case}\label{sssec:real}
When the Lagrangian is the fixed locus of an anti-symplectic involution, the cohomological assumptions can be relaxed. As in~\cite{ST2}, this is because invariance under involution forces some of the contributions to the Maurer-Cartan equation to cancel in pairs, resulting in the vanishing of the corresponding obstruction classes.

Let $\phi:X\to X$ be an anti-symplectic involution, that is, $\phi^*\omega=-\omega.$ Let $L\subset \fix(\phi)$ and $J\in\J_\phi$. In particular, $\phi^*J=-J$.
For the entire Section~\ref{sec:real}, we take $\ssly_L \subset  H_2(X,L;\Z)$ with $\Im(\Id+\phi_*) \subset \ssly_L,$ so $\phi_*$ acts on $\sly_L = H_2(X,L;Z)/\ssly_L$ as $-\Id.$ Also, we take $\deg t_j\in 2\Z$ for all $j=0,\ldots,N$.
We denote by $\Hh_\phi^{even}(X,L;\R)$ (resp. $H_\phi^{even}(X;\R)$) the direct sum over $k$ of the $(-1)^{k}$-eigenspace of $\phi^*$ acting on $\Hh^{2k}(X,L;\R)$ (resp. $H^{2k}(X;\R)$). Note that the Poincar\'e duals of $\phi$-invariant almost complex submanifolds of $X$ disjoint from $L$ belong to $\Hh_\phi^{even}(X,L;\R).$
Extend the action of $\phi^*$ to $\L,Q,R,C,$ and $D,$ by taking
\[
\phi^*T^\beta = (-1)^{\mu(\beta)/2}T^\beta, \qquad \phi^* t_i = (-1)^{\deg t_i/2}t_i, \qquad \phi^* s = -s.
\]
Elements $a \in \L,Q,R,C,D,$ are called \textbf{real} if
\begin{equation}\label{eq:relt}
\phi^* a = -a.
\end{equation}
A pair of elements as above $(a_1,a_2)$ is real if both components $a_1,a_2,$ are real.
For a group $Z$ on which $\phi^*$ acts, let $Z^{-\phi^*}\subset Z$ denote the elements fixed by $-\phi^*.$

Define a subset of $\{0,\ldots,n-1\}$ as follows:
\[
\mP(n) := \{ j\; |\; j\equiv 0,3,n,n+1\pmod 4\}.
\]
In other words,
\[
\mP(n)=\begin{cases}
\{j\; |\; j\equiv 0,3 \pmod 4\}, & n\equiv 3\pmod 4,\\
\{j\; |\; j\not\equiv 1 \pmod 4\}, & n\equiv 2\pmod 4,\\
\{j\; |\; j\not\equiv 2 \pmod 4\}, & n\equiv 0\pmod 4.
\end{cases}
\]
Define also
\[
\mQ(n) := \{ j\; |\; j\equiv 0,n+1\pmod 4\} \; \subset \mP(n).
\]

Assume $i^*:H^j(X;\R)\to H^j(L;\R)$ is surjective for all $j\in \mP(n)$.
Let $\theta_0,\ldots,\theta_r\in A^*(X)$ be closed, homogeneous forms such that $i^*[\theta_j]$ form a basis for $\oplus_{j\in\mQ(n)}H^{j}(L;\R)$.
We proceed with the assumption that $[L]=0\in H_n(X;\R)$, and the choice of $\zeta\in \A^{n-1}(X;\R)$ such that $d\zeta = i_*1$.

As before, the following map is well defined.
\begin{gather*}
\Xi_\phi: \{\text{separated real bounding pairs}\}/\sim\quad \lrarr \quad (\mI_Q\Hh^*(X,L;Q))^{-\phi^*}_2\oplus \bigoplus_{m=0}^r (\mI_R)_{1-n+|\theta_m|}, \\
\Xi_\phi(b,\gamma)=
[\gamma]\oplus \bigoplus_{m=0}^r(\pbg_{0,0}-(-1)^n\qg_{\emptyset,1}(\zeta))(\theta_m).
\end{gather*}

\begin{thm}[Classification -- real spin case]\label{thm:B} 
Suppose $(X,L,\omega,\phi)$ is a real setting such that $\s$ is induced by a spin structure and $n \not \equiv 1 \pmod 4$.
Assume $i^*:H^j(X;\R)\to H^j(L;\R)$ is surjective for all $j\in \mP(n)$.
Then the map $\Xi_\phi$ is bijective.
\end{thm}

\begin{rem}
By Lemma~\ref{lm:reR} below, the range of $\Xi_\phi$ consists of real elements of positive valuation. Thus, it is the classifying space we would expect to see as a real analog of the range of $\Xi$ in Theorem~\ref{thm:A}.
\end{rem}

\begin{rem}
It was already noted in~\cite[Remark 1.3]{ST2} that when $n=2,3,$ the cohomological assumptions are automatically satisfied. In these dimensions, the formulation of our theorem reduces to the one in~\cite{ST2}, with $r=0$ and $\theta_0=1\in A^0(X)$. Further, any bounding chain is a (closed) $n$-form that can be thought of as a multiple of a Poincar\'e dual of a point.
\end{rem}

\subsection{Context}
\subsubsection{Related results}\label{sssec:related}

The book~\cite{FOOO} introduces a number of techniques in obstruction theory, that culminate in numerous results valid in various setups. Two results particularly relevant to the one in this paper are the following. In a special case of Theorem 3.8.11, the authors use $\p$ operators to show vanishing of the push-forward obstruction classes, essentially $i_*[o_j]$, from which the existence of a bounding chain is deduced (without bulk deformation), under topological assumptions similar to ours. In a special case of Theorem 3.8.41, the authors use $\q$ operators to construct a bounding pair by simultaneously correcting the bulk deformation and the bounding chain, again under similar topological assumptions.

The results in the current paper are different in flavor: Not only do we wish to establish the existence of a bounding pair, but we need to understand the structure of the space of all such pairs. In particular, the question of when two pairs are gauge-equivalent is of the utmost importance.
This is a generalization of~\cite{ST2}.


\subsubsection{Application in open Gromov-Witten theory}

In~\cite[Remark 4.17]{ST3} we explain why, generally, all genus zero open Gromov-Witten invariants can be obtained from point-like bounding chains.
So, whenever available and if unique, a point-like chain is a canonical choice for the purpose of defining Gromov-Witten invariants. In the settings of both Theorems~\ref{thm:A} and~\ref{thm:B} such a choice is indeed available and unique.

So, in cases that satisfy the assumptions of Theorem~\ref{thm:A} (resp. Theorem~\ref{thm:B}), open Gromov-Witten invariants are canonically defined, as follows.
Let $\gamma_0,\ldots,\gamma_N\in \Ah^*(X,L)$ be closed forms representing a basis of $\Hh^*(X,L;\R)$, and set $\gamma=\sum_{j=0}^Nt_j\gamma_j$. For simplicity take $\gamma_0=1_X$.
Then by Theorem~\ref{thm:A} (resp. Theorem~\ref{thm:B}) there exists a separated (resp. separated real) bounding chain $b$ for $\mg$ such that $\Xi(\gamma,b)$ (resp. $\Xi_\phi(\gamma,b)$) equals $([\gamma],\vec{a}=(a_m)_{m=0}^r)$ , with
\[
a_m=\begin{cases} s, & m=0,\\ 0, & m>0.\end{cases}
\]
Moreover, such $b$ is unique up to gauge equivalence.
Use this pair $(\gamma,b)$ to construct the enhanced open superpotential $\Ob=\Ob(\gamma,b)$ that generates the invariants, via the procedure given in~\cite{ST3}.

\subsection{Acknowledgements}
The idea of this paper came about in conversations with J. Solomon during the preparation of~\cite{ST2}. The idea of using $\q_{\emptyset}(\zeta)$ to correct $\p$ was taught to me by P. Giterman and J. Solomon, and was first used in~\cite{GST}. I am also grateful to K. Fukaya, P. Giterman, and K. Ono for generally helpful conversations.
While working on the manuscript, I was partly supported by NSF grant no. DMS-163852, ISF grant no. 2793/21, and the Colton Foundation.

\section{Background}\label{sec:background}

\subsection{Forms and currernts}\label{ssec:currconv}

Here we set our conventions for differential forms and currents on orbifolds. Details are worked out in~\cite{ST4}.

\subsubsection{Conventions and notation}\label{sssec:curconv}

As in Section~\ref{sssec:rings}, for a compact manifold $M$ the space $\A^k(M)$ of currents of cohomological degree $k$ is the continuous dual space of $A^{\dim M - k}(M)$ with the weak-$\ast$
topology.
We identify $A^k(M)$ as a subspace of $\A^k(M)$ via
\begin{gather*}
\varphi:A^k(M)\hookrightarrow \A^k(M),\\
\varphi(\gamma)(\eta)=\int_{M}\gamma\wedge\eta,\quad \eta\in A^{\dim M -k}(M).
\end{gather*}
Exterior derivative extends to currents via $d\a(\eta)=(-1)^{1+|\a|}\a(d\eta)$.

For $f:M\to N$, set $rdim f:=\dim N-\dim M$.
The push-forward along $f$,
\[
f_* : \A^k(M) \to \A^{k-rdim f}(N),
\]
is defined via
\[
(f_*\a)(\xi)=(-1)^{m\cdot rdim f}\a(f^*\xi),\qquad \xi\in A^m(N).
\]
So, when $f$ is a submersion, $f_* \varphi(\a) = \varphi(f_*\a).$
Similarly, for $f:M\to N$ a relatively oriented submersion, define the pull-back
\[
f^* : \A^{k}(N) \to \A^{k}(M)
\]
by
\[
(f^*\a)(\xi)=\a(f_*\xi),\qquad \xi\in A^m(N).
\]
We refer to~\cite[Section 6]{ST4} for full details and proofs of properties of currents on (orbifolds, and in particular) manifolds with corners. The bottom line is that currents and differential forms satisfy very similar properties, and $d$, pull-back and push-forward respect the module structure of currents over forms. The one caveat is that forms can always be pulled back but only push forward along proper submersions, while currents are dual and accordingly always push forward, but only pull back along proper submersions.

For convenience, we use the following notation:
\[
\langle \a,\xi\rangle_M :=(-1)^{|\xi|}\a(\xi)
= (-1)^{\dim M-k}\a(\xi),
\qquad \a\in \A^k(M), \xi\in A^{\dim M-k}(M).
\]

\subsubsection{Pairings}
For $M=L,X,$ define a pairing on $A^*(M)$ via
\[
\langle \xi,\eta\rangle_M=(-1)^{|\eta|}\int_M\xi\wedge\eta,\quad \forall \xi,\eta\in A^*(M).
\]
By Poincar\'e duality, the above pairing descends to a non-degenerate pairing on cohomology.
Note that
\begin{equation}\label{eq:pairsymm}
\langle\xi,\eta\rangle_M
=(-1)^{(|\eta|+1)(|\xi|+1)+1}\langle\eta,\xi\rangle_M.
\end{equation}

It is easy to verify that for any proper sumbersion $f:M\to N$, we have
\[
\langle \eta, f_*\alpha\rangle_N = (-1)^{rdim f}\langle f^*\eta,\alpha\rangle_N,
\qquad \eta \in A^*(N), \xi\in A^*(M).
\]
We generalize our notation for currents accordingly and use the same formula for $f$ that is not necessarily a submersion:
\[
\langle \xi, f_*\alpha\rangle_M :=(-1)^{rdim f}\langle f^*\xi, \alpha\rangle_M,
\]
where the right-hand side is understood as pairing of differential forms. 
In particular, we have the following.

\begin{lm}[{\cite[Lemma 4.3]{ST3}}]\label{lm:pairing}
For all $\alpha \in A^*(L;R)$, $\eta\in A^*(X;R)$, we have
\[
\langle \eta, i_*\alpha\rangle_X = (-1)^n\cdot\langle i^*\eta,\alpha\rangle_L.
\]
\end{lm}
In combination with~\eqref{eq:pairsymm}, this agrees with the definition of push-forward:
\[
\langle i_*\alpha,\eta\rangle_X = (-1)^{n|\eta|}\cdot\langle \alpha, i^*\eta\rangle_L.
\]

The next result was relied on in the definition of $\Xi$ and $\Xi_\phi$ in the main theorems.
\begin{lm}\label{lm:exact}
If $[L]=0\in H_n(X;\R)$, then the current $i_*1$ is exact.
\end{lm}

\begin{proof}
	By assumption, there exists a rational weighted manifold $V$ with $\d V = L$ and a map $i_V : V \to X$ such that $i_V|_{\d V} = i.$ To see how $i_*1$ acts on forms (as a current), pair it with $\eta\in A^n(X)$:
		\[
		\int_X i_*1\wedge\eta=\int_Li^*\eta =\int_{\d V}i_V^*\eta=\int_Vd(i_V^*\eta)=\int_Vi_V^*d\eta=-\int_Xd ((i_V)_*1_V)\wedge\eta.
		\]
		So, $i_*1=d(-(i_V)_*1_V)$.
		
\end{proof}

\subsubsection{Pairings over a dga}\label{sssec:dga}

When discussing $A_\infty$ structures and open-closed on a pseudo-isotopy, we work over the dga $\mR=A^*([0,1])$. Thus, for $M=L,X,$ we need to define a pairing on $A^*(I\times M;R)$ over $\mR$.
Let $p^M:I\times M\to I$ be the projections. Define the two pairings
\[
\ll\;,\;\gg_M : A^*(I\times M)\times A^*(I\times M) \to \mR
\]
via
\[
\ll \xit,\etat\gg_M=(-1)^{|\etat|}(p^M)_*(\xit\wedge\etat)
\]
and extend linearly to allow coefficients in $R$.
We also use the notation of a pairing where one of the inputs is a current, similarly to the convention set up in Section~\ref{sssec:curconv} above.

The following two lemmas describe how these pairings interact with exterior derivative.
\begin{lm}[{\cite[Lemma 4.9]{ST1}}]\label{lm:d_ll_gg}
For $\at,\at'\in \mC,$ we have
\[
(-1)^{|\at|+|\at'|+n}\int_I d\ll\at,\at'\gg=\langle j_1^*\at,j_1^*\at'\rangle- \langle j_0^*\at,j_0^*\at'\rangle.
\]
\end{lm}

\begin{lm}[{\cite[Lemma 4.11]{ST1}}]\label{lm:qt_cyclic}
\[
\ll d\at_1,\at_2\gg
=
d\ll\at_1,\at_2\gg+(-1)^{(|\at_1|+1)(|\at_2|+1)}\ll d\at_2,\at_1\gg.
\]
\end{lm}
The two above lemmas hold equally well for $\at,\at',\at_j\in A^*(I\times X;R)$ instead of $\mC$.

\subsection{Filtrations}\label{ssec:filtration}

Recall the valuation $\nu:R\lrarr \R_{\ge 0}$ defined Section~\ref{sssec:rings}, and recall that $\mI_R$ is the ideal of positive valuation elements in $R$.
For $\Ups=C,\mC,\A^*(X;R),\A^*(I\times X;R)$,
denote by $F^{E}$ the filtration induced by $\nu$ on $\Ups$. That is,
\[
\lambda\in F^{E}\Ups\iff \nu(\lambda)> E.
\]
We use the following adaptation of the gapped condition of~\cite{FOOO}.
\begin{dfn}
A multiplicative submonoid $G\subset R$ is said to be \sababa{} if it can be written as a list
\begin{equation}\label{eq:list}
G=\{\pm\lambda_0=\pm T^{\beta_0}, \pm\lambda_1,\pm\lambda_2,\ldots\}
\end{equation}
such that $i<j\Rarr \nu(\lambda_i)\le \nu(\lambda_j).$
\end{dfn}

For $j=1,\ldots,m,$ and elements $\alpha_j=\sum_i\lambda_{ij}\alpha_{ij}\in \Ups$ decomposed so that $\lambda_{ij}\in R$, $\nu(a_{ij})=0$, denote by $G(\alpha_1,\ldots,\alpha_m)$ the multiplicative monoid generated by $\{\pm T^\beta\,|\,\beta\in \sly_\infty\}$, $\{t_j\}_{j=0}^N$, and  $\{\lambda_{ij}\}_{i,j}$.

\begin{lm}[{\cite[Lemma 3.3]{ST2}}]\label{lm:sababa}
For $\alpha_1,\ldots,\alpha_m\in \mI_R$, the monoid $G(\alpha_1,\ldots,\alpha_m)$ is \sababa{}.
\end{lm}

For $\alpha_1,\ldots,\alpha_m\in\mI_R$, write the image of $G=G(\{\alpha_j\}_j)$ under $\nu$ as a sequence
\[
\nu(G)=\{E_0=0,E_1,E_2,\ldots\}
\]
such that $E_i<E_{i+1}$.
Let $\kappa_i\in\Z_{\ge 0}$ be the largest index such that $\nu(\lambda_{\kappa_i})=E_i.$ In particular, the set $\{\pm\lambda_j\}_{j= \kappa_i+1}^{\kappa_{i+1}}$ is exactly the set of the elements in $G$ of valuation $E_i$.

\subsection{Operators and bounding pairs}\label{ssec:pq}

\subsubsection{Closed operators}\label{sssec:clq}

Let $\gamma \in \mI_Q \Ah^*(X,L;Q)$ with $d\gamma =0$ and $\deg \gamma =2$.
Recall that in Section~\ref{sssec:intro_clq} we defined operators
\[
\qg_{\emptyset,l}:A^*(X;Q)^{\otimes l}\lrarr A^*(X;Q)
\]
via
\[
\qg_{\emptyset,l}(\eta_1,\ldots,\eta_l):=
\sum_{\substack{\beta\in H_2(X;\Z)\\m\ge 0}}
(-1)^{w_\s(\beta)}
\frac{1}{m!}T^{\varpi(\beta)}{ev_0^\beta}_* (\bigwedge_{j=1}^l(ev_j^\beta)^*\eta_j\wedge \bigwedge_{j=l+1}^{l+m}(ev_j^\beta)^*\gamma).
\]
Using the same formula, the definition of $\qg_{\emptyset,l}$ can be extended to the situation where one of the inputs $\eta_j$ is a current. To do so, we relabel the marked points so that the evaluation at that index is a proper submersion (then $ev_0$ might no longer be so). Then pullback of a current along this evaluation is well defined, the expression $\bigwedge_{j=1}^l(ev_j^\beta)^*\eta_j\wedge \bigwedge_{j=l+1}^{l+m}(ev_j^\beta)^*\gamma$ is to be understood as a current, and the output of $\qg_{\emptyset,l}$ is again a current. By slight abuse of notation, we implicitly allow this possibility hereafter.

The following lemmas are well known. We use the formulation given in~\cite[Section 2.3]{ST3}.

\begin{lm}[Closed unit]\label{lm:cunit}
For $\gamma_1,\ldots,\gamma_{l-1} \in A^*(X),$
\[
\q_{\emptyset,l}^\beta(1,\gamma_1,\ldots,\gamma_{l-1}) =
\begin{cases}
\gamma_1, & \beta = 0 \text{ and } l = 2, \\
0, & \text{otherwise}.
\end{cases}
\]
\end{lm}

\begin{lm}[Closed zero energy]\label{lm:czero}
For $\gamma_1,\ldots,\gamma_l \in A^*(X),$
\[
\q_{\emptyset,l}^0(\gamma_1,\ldots,\gamma_l) =
\begin{cases}
\gamma_1 \wedge \gamma_2, & l = 2, \\
0,  & \text{otherwise.}
\end{cases}
\]
\end{lm}

The following we prove in order to verify the sign.
\begin{lm}[Closed cyclic property]\label{lm:qemptcyc}
\[
\langle\qg_{\emptyset,l}(\eta_1,\ldots,\eta_l), \eta_{l+1}\rangle_X
=
(-1)^{|\eta_{l}|+|\eta_{l+1}|\big(\sum_{j=1}^l|\eta_j|+1\big)}
\langle\qg_{\emptyset,l}(\eta_{l+1},\eta_1,\ldots,\eta_{l-1}), \eta_{l}\rangle_X.
\]
\end{lm}

\begin{proof}
For the non-deformed operator in a fixed degree $\q^\beta_{\emptyset,l}$, use the fact that $rdim (ev_0)$ is even to compute
\begin{align*}
(-1)^{w_{\s}(\beta)}\langle\q^\beta_{\emptyset,l}(\eta_1,\ldots,\eta_l), &\eta_{l+1}\rangle_X
=
(-1)^{|\eta_{l+1}|}
pt_*\big((ev^\beta_0)_*(\bigwedge_{j=1}^l(ev_j^\beta)^*\eta_j)\wedge \eta_{l+1}\big)\\
&=
(-1)^{|\eta_{l+1}|+|\eta_{l+1}|\sum_{j=1}^l|\eta_j|}
pt_*\big((ev_0^\beta)^*\eta_{l+1}\wedge \bigwedge_{j=1}^{l}(ev_j^\beta)^*\eta_j\big)\\
&=
(-1)^{|\eta_{l+1}|+|\eta_{l+1}|\sum_{j=1}^l|\eta_j|}
pt_*\big((ev_1^\beta)^*\eta_{l+1}\wedge \bigwedge_{j=2}^{l-1}(ev_j^\beta)^*\eta_j\wedge (ev_{l}^\beta)^*\eta_{l}\big)\\
&=
(-1)^{|\eta_{l}|+|\eta_{l+1}|\big(\sum_{j=1}^l|\eta_j|+1\big)}
(-1)^{w_{\s}(\beta)}
\langle\q^\beta_{\emptyset,l}(\eta_{l+1},\eta_1,\ldots,\eta_{l-1}), \eta_{l}\rangle_X.
\end{align*}
For the deformed $\qg_{\emptyset, l}$ the argument is the same. Note that, since $\deg \gamma$ is even, the signs remain unaffected.
\end{proof}
Of particular interest to us is the special case
\[
\langle\qg_{\emptyset, 1}(\zeta),\xi\rangle_X = (-1)^{|\zeta|+|\xi|+|\zeta||\xi|}\langle \qg_{\emptyset, 1}(\xi),\zeta\rangle_X.
\]

\subsubsection{Closed operators on a pseudo-isotopy}\label{sssec:clqt}

Write for short
\[
I:=[0,1]\subset \R.
\]
Similarly to Section~\ref{sssec:clq}, we define sphere-based operators
\[
\qt_{\emptyset,l}:A^*(I\times X;R)^{\otimes l}\lrarr A^*(I\times X;R)
\]
for a pseudo-isotopy as follows. For $\beta\in H_2(X;\Z)$ let
\[
\Mt_{l+1}(\beta):=\{(t,u,\vec{w})\;|\,(u,\vec{w})\in \M_{l+1}(\beta;J_t)\}.
\]
For $j=0,\ldots,l,$ let
\begin{gather*}
\evt_j^\beta:\Mt_{l+1}(\beta)\to I\times X,\\
\evt_j^\beta(t,u,\vec{w}):=(t,u(w_j)),
\end{gather*}
be the evaluation maps. Assume that all the moduli spaces $\Mt_{l+1}(\beta)$ are smooth orbifolds and $\evt_0$ is a submersion. Again let $w_{\s} \in H^2(X;\Z/2\Z)$ be the class with $w_2(TL) = i^* w_{\s}$ determined by the relative spin structure $\s$.
For $l\ge 0$, $(l,\beta)\ne (1,0),(0,0)$, set
\[
\qt_{\emptyset,l}^\beta(\gt_1,\ldots,\gt_l):=
(-1)^{w_\s(\beta)}
(\evt_0^\beta)_*(\wedge_{j=1}^l(\evt_j^\beta)^*\gt_j),
\]
and define $\qt_{\emptyset,1}^0:= 0, \qt_{\emptyset,0}^0:= 0,$ and
\[
\qt_{\emptyset,l}(\gt_1,\ldots,\gt_l):=
\sum_{\beta\in H_2(X)}T^{\pr(\beta)} \qt_{\emptyset,l}^\beta(\gt_1,\ldots,\gt_l).
\]

Let $p^X:I\times X \to X$ be the projection, and recall from Section~\ref{sssec:dga} the pairing
\[
\ll\,,\,\gg_X:A^*(I\times X)\otimes A^*(I\times X)\lrarr \mathfrak{R}
\]
defined by
\[
\ll\tilde{\xi},\tilde{\eta}\gg_X:=(-1)^{|\etat|}(p^X)_*(\tilde{\xi}\wedge\tilde{\eta}).
\]
The following is proved similarly to Lemma~\ref{lm:qemptcyc}
\begin{lm}[Closed cyclic property on pseudo-isotopy]\label{lm:qtemptcyc}
\[
\ll\qg_{\emptyset,l}(\eta_1,\ldots,\eta_l), \eta_{l+1}\gg_X
=
(-1)^{|\eta_{l}|+|\eta_{l+1}|\big(\sum_{j=1}^l|\eta_j|+1\big)}
\ll\qg_{\emptyset,l}(\eta_{l+1},\eta_1,\ldots,\eta_{l-1}), \eta_{l}\gg_X.
\]
\end{lm}

\subsubsection{Open operators}\label{sssec:m}

Let $\gamma\in \mI_QD$ be a closed form with $\deg_D\gamma=2$.
Recall from Section~\ref{sssec:intro_m} that structure maps
\[
\mg_k:C^{\otimes k}\lrarr C
\]
are defined by
\begin{multline*}
\m^{\gamma}_k(\alpha_1,\ldots,\alpha_k):=\\
=\delta_{k,1}\cdot d\alpha_1+(-1)^{\sum_{j=1}^kj(|\alpha_j|+1)+1}
\sum_{\substack{\beta\in\sly\\l\ge0}}T^{\beta}\frac{1}{l!}{evb_0^\beta}_* (\bigwedge_{j=1}^l(evi_j^\beta)^*\gamma\wedge
\bigwedge_{j=1}^k (evb_j^\beta)^*\alpha_j
).
\end{multline*}
The condition $\gamma\in \mI_QD$ ensures that the infinite sum converges.
In~\cite{FOOO,Fukaya,ST1}, it is shown that $(C,\{\mg_k\}_{k\ge 0})$ is an $A_\infty$ algebra, , that is, for all $\a_1,\ldots,\a_k\in A^*(L;R)$, we have
\[
\sum_{\substack{k_1+k_2=k+1\\ 1\le i\le k_1}}\mg_{k_1}(\a_1,\ldots,\a_{i-1},\mg_{k_2}(\a_i,\ldots, \a_{i+k_2-1}),\a_{i+k_2},\ldots, \a_k) = 0.
\]

Another type of an open operator, defined in~\cite{ST1} based on $\m_{-1}$ of~\cite{Fukaya}, plays an auxiliary role in some of the proofs:
For $l\ge 0$, define
\[
\qbg_{-1,l}:A^*(X;Q)^{\otimes l}\lrarr R
\]
by
\begin{gather}\label{eq:q-1def}
\qbg_{-1,l}(\eta_1\otimes\cdots\otimes\eta_l):=
\langle \mg(e^b),b\rangle_L
+\sum_{\substack{\beta\in\sly\\m\ge 0}}\frac{1}{m!}T^\beta\int_{\M_{0,l}(\beta)} \bigwedge_{j=1}^l (evi_j^\beta)^*\eta_j\wedge\bigwedge_{j=l+1}^{l+m}(evi_j^\beta)^*\gamma.
\end{gather}
Of particular use to us will be $\qbg_{-1,0}$ and $\qbg_{-1,1}$. The latter satisfies a useful equation, to formulate which we need additional notation.
Let $i:L\to X$ be the inclusion and define
\[
\i:A^*(X;R)\to R
\]
by
\begin{equation}\label{eq:i}
\i(\eta)=(-1)^{n+|\eta|}\int_Li^*\eta.
\end{equation}
\begin{lm}[{\cite[Lemma 4.1, Corollary 4.2]{ST3}}]\label{lm:ceq}
For any $\eta\in A^*(X;R),$
we have
\[
(-1)^{n+1}\qbg_{-1,1}(d\eta) -c\cdot \i(\eta)=
\i(\qg_{\emptyset,1}(\eta)).
\]
In particular, if $\eta$ is closed, then
\[
\i(\qg_{\emptyset,1}(\eta))+c\cdot \i(\eta)=0.
\]
\end{lm}

The following result is convenient to keep in mind, though not strictly needed for our argument.
\begin{lm}\label{lm:p0cl}
The map $\i$ from~\eqref{eq:i} satisfies, for all $\eta$,
\begin{gather*}
\langle i_*1,\eta\rangle_X = \i(\eta), \\
\langle\qg_{\emptyset,1}(i_*1),\eta\rangle_X
=\i(\qg_{\emptyset,1}(\eta)).
\end{gather*}
If moreover $\eta$ is closed and $c$ satisfies~\eqref{eq:bc}, then
\[
\langle\qg_{\emptyset,1}(i_*1),\eta\rangle_X = - c\cdot\langle i_*1,\eta\rangle_X
\]
\end{lm}

\begin{proof}
For the first identity, observe that both sides vanish $|\eta|\equiv n \pmod 2$. Thus,
\[
\langle i_*1,\eta\rangle_X
= (-1)^{n|\eta|}\langle 1, i^*\eta\rangle_L = \i(\eta).
\]
For the second identity, use Lemma~\ref{lm:qemptcyc} to compute
\begin{align*}
\langle\qg_{\emptyset,1}(i_*1),\eta\rangle_X
&=
(-1)^{n+|\eta|+n|\eta|}
\langle\qg_{\emptyset,1}(\eta),i_*1\rangle_X\\
&=
(-1)^{|\eta|+n|\eta|}
\langle i^*\qg_{\emptyset,1}(\eta),1\rangle_L\\
&=
(-1)^{|\eta|+n|\eta|}
\int_L i^*\qg_{\emptyset,1}(\eta)\\
&=
(-1)^{n|\eta|+n}
\i(\qg_{\emptyset,1}(\eta)).
\end{align*}
Again using $|\eta|\equiv n\pmod 2$, we conclude
\[
\langle\qg_{\emptyset,1}(i_*1),\eta\rangle_X = \i(\qg_{\emptyset,1}(\eta)).
\]
If in addition $\eta$ is closed, Lemma~\ref{lm:ceq} gives
\[
\i(\qg_{\emptyset,1}(\eta))+c\cdot \i(\eta)=0
\]
and the result follows.
\end{proof}

\subsubsection{Open operators on a pseudo-isotopy}\label{sssec:mt}

Set
\begin{equation}\label{eq:mC}
\mC:=A^*(I\times L;R),\mbox{ and }\mD:=\Ah^*(I\times X,I\times L;Q),
\end{equation}
again completed with respect to $\nu$, while we use $\nu$ also to denote the induced valuation.
We construct a family of $A_\infty$ structures on $\mC$. Fix a family of $\omega$-tame almost complex structures $\{J_t\}_{t\in I}$.
For each $\beta, k, l,$ set
\[
\Mt_{k+1,l}(\beta):=\{(t,\uu)\,|\,\uu\in\M_{k+1,l}(\beta;J_t)\}.
\]
The moduli space $\Mt_{k+1,l}(\beta)$ comes with evaluation maps
\begin{gather*}
\evbt_j:\Mt_{k+1,l}(\beta)\lrarr I\times L, \quad j\in\{0,\ldots,k\},\\
\evbt_j(t,(\Sigma,u,\vec{z},\vec{w})):=(t,u(z_j)),
\end{gather*}
and
\begin{gather*}
\evit_j:\Mt_{k+1,l}(\beta)\lrarr I\times X, \quad j\in\{1,\ldots,l\},\\
\evit_j(t,(\Sigma,u,\vec{z},\vec{w})):=(t,u(w_j)).
\end{gather*}
As with the usual moduli spaces, we assume all $\Mt_{k+1,l}(\beta)$ are smooth orbifolds with corners, and $\evbt_0$ is a proper submersion. As discussed in~\cite[Example 4.1]{ST1}, in the special case when $J_t$ is constant in $t$, this follows from the regularity assumptions set in~\ref{sssec:regularity}. In general, this is a non-trivial requirement.

We define maps
\[
\mgt_k:\mC^{\otimes k}\lrarr \mC
\]
via
\begin{multline*}
\mgt_k(\alpha_1,\ldots,\alpha_k):=\\
=\delta_{k,1}\cdot d\alpha_1+(-1)^{\sum_{j=1}^kj(|\alpha_j|+1)+1}
\sum_{\substack{\beta\in\sly\\l\ge0}}T^{\beta}\frac{1}{l!}(\evbt_0^\beta)_* (\bigwedge_{j=1}^l(\evit_j^\beta)^*\gamma\wedge
\bigwedge_{j=1}^k (\evbt_j^\beta)^*\alpha_j
).
\end{multline*}
Similarly to the usual operators, these give an $A_\infty$ structure on $\mC$.

\subsubsection{Bounding pairs}\label{sssec:bdpair}

Write for short
\[
\mg(e^b):=\sum_{k\ge 0}\m_k^\gamma(b^{\otimes k}).
\]
Recall that in Section~\ref{sssec:intro_bdpair} we defined $(\gamma,b)$ to be a bounding pair if it satisfies
\[
\mg(e^b)=c\cdot 1, \qquad c\in \mI_R,\;\deg_Rc=2.
\]
In this situation, we called $b$ a bounding chain for $\mg$.

We say a bounding pair $(\gamma,b)$ with respect to $J$ is \textbf{gauge equivalent} to a bounding pair $(\gamma',b')$ with respect to $J'$, if there exist
a path $\{J_t\}$ in $\mathcal{J}$ from $J$ to $J'$,
$\gt\in (\mI_Q\mD)_2$, and $\bt\in (\mI_R\mC)_1$ such that
\begin{gather}
j_0^*\gt=\gamma,\quad j_1^*\gt=\gamma',\quad  d\gt=0,\notag\\
j_0^*\bt=b,\quad j_1^*\bt=b',\notag \\
\mgt(e^{\bt})=\ct\cdot 1,\qquad \ct\in (\mI_R\mR)_2, \qquad d\ct = 0.\label{eq:mct}
\end{gather}
In this case, we say that $(\mgt,\bt)$ is a pseudoisotopy from $(\mg,b)$ to $(\m^{\gamma'},b')$ and write $(\gamma,b)\sim(\gamma',b')$. In the special case $J_t= J = J',$ $\gamma=\gamma'$, and $\gt=\pi^*\gamma$, we say $b$ is gauge equivalent to $b'$ as a bounding chain for $\mg$.

The following is proved in \cite[Remark 3.13]{ST2}.
\begin{lm}\label{lm:cinvt}
Suppose there exists a pseudoisotopy from $(\mg,b)$ to $(\m^{\gamma'},b')$. Then $\mg(e^b)=\m^{\gamma'}(e^{b'})$.
\end{lm}

To close the section, we define the analog on pseudo-isotopy of the operator $\q_{-1}$. Namely, let
\[
p^L:I\times L\lrarr I,\qquad p_\M: \Mt_{k+1,l}(\beta)\lrarr I,
\]
denote the projections, and recall from Section~\ref{sssec:dga} the pairing
\[
\ll\,,\,\gg_L:\mC\otimes\mC\lrarr \mathfrak{R}
\]
defined by
\[
\ll\tilde{\xi},\tilde{\eta}\gg_L:=(-1)^{|\etat|}(p^L)_*(\tilde{\xi}\wedge\tilde{\eta}).
\]
For $l\ge 0$, define
\[
\qtbg_{-1,l}:A^*(I\times X;Q)^{\otimes l}\lrarr \mR
\]
by
\begin{gather*}
\qbgt_{-1,l}(\eta_1\otimes\cdots\otimes\eta_l):=
\ll \mgt(e^{\bt}),\bt\gg_L
+\sum_{\substack{\beta\in\sly\\m\ge 0}}\frac{1}{m!}T^\beta (p_\M)_* \bigwedge_{j=1}^l (\evit_j^\beta)^*\eta_j\wedge\bigwedge_{j=l+1}^{l+m}(\evit_j^\beta)^*\gt.
\end{gather*}

\subsubsection{Open-closed maps}\label{sssec:p}

We define a $\p$ operator in the spirit of~\cite{FOOO}.
The particular definition below and some of the properties cited are worked out in detail in~\cite{GST}.

Relabel the marked points on the space $\M_{k,l+1}(\beta)$ to get
evaluation maps
\[
evb_j^\beta:\M_{k,l+1}(\beta)\to L, \quad j=1,\ldots,k,
\]
and
\[
evi_j^\beta:\M_{k,l+1}(\beta)\to X, \quad j=0,\ldots,l.
\]
The map $evi_0$ is not necessarily (in fact, hardly ever) a proper submersion. So, push-forward along $evi_0$ is to be understood as a current.

For all $\beta\in\sly$, $k,l\ge 0$,  $(k,l,\beta) \ne (0,0,\beta_0)$, define
\[
\p_{k,l}^\beta:A^*(L;R)^{\otimes k}\otimes A^*(X;Q)^{\otimes l} \lrarr \A^*(X;R),
\]
for $\a=(\a_1,\ldots, \a_k), \gamma
=(\gamma_1,\ldots, \gamma_l)$, by
\begin{align*}
\p^{\beta}_{k,l}(\a;\gamma):=
(-1)^{\varepsilon_p(\a)}
(evi_0^\beta)_* \left(evi^*\gamma\wedge evb^*\a\right)
\end{align*}
with
\[
\varepsilon_p(\a):=
\sum_{j=1}^k(n+j)(|\alpha_j|+1).
\]
Define also $\p_{0,0}^{\beta_0}=0.$
Set
\begin{align*}
\pkl:=\sum_{\beta\in\sly}T^{\beta}\pkl^{\beta}.
\end{align*}

For for a list $\a=(\alpha_1,\ldots,\alpha_k)$ and a cyclic permutation $\sigma\in \Z/k\Z$, denote by $\a^\sigma$ the list
\[
\a^\sigma:=(\alpha_{\sigma(1)},\ldots,\alpha_{\sigma(k)}),
\]
and define
\[
s_\sigma^{[1]}(\a)=\sum_{\substack{i>j\\ \sigma(i)<\sigma(j)}}(|\alpha_{\sigma(i)}|+1)(|\alpha_{\sigma(j)}|+1).
\]

Let $\gamma \in \mI_Q A^*(X;Q)$ such that $|\gamma|=2$ and $d\gamma=0$. Let $b\in \mI_R A^*(L;R)$ such that $|b|=1$.
Define, for $k,l\ge 0$, the deformed open-closed maps by
\begin{equation}\label{eq:pg}
\pg_{k}(\a):=\sum_{t\ge 0}
\frac{1}{t!}
\p_{k,l+t}(\a_1\otimes \cdots \otimes \a_k ;\gamma\otimes \gamma^{\otimes t}),
\end{equation}
and
\begin{equation}\label{eq:pbg}
\pbg_{k,l}(\a;\gamma):=\sum_{s,t\ge 0}
\sum_{\sum_{j=0}^{k-1}i_j=s}
\frac{1}{t!}
\p_{k+s,l+t}(b^{\otimes i_0}\otimes \a_1\otimes b^{\otimes i_1}\otimes \cdots \otimes b^{\otimes i_{k-1}}\otimes \a_k ;\gamma\otimes \gamma^{\otimes t}).
\end{equation}

\begin{lm}[{\cite[Proposition 4.2]{GST}}]\label{lm:pstructure}
Consider a list $\a=(\alpha_1,\ldots,\alpha_k),$ $\alpha_j\in A^*(L;R)$, and $\gamma\in (\mI_QA^*(X;Q))_2.$
Then
\begin{equation}\label{eq:p_rel}
d\pg_{k}(\a)
=
\sum_{\sigma\in\Z/k\Z}
(-1)^{s_\sigma^{[1]}(\a)+n+1}
\pg(\mg ((\a^\sigma)_{1}) \otimes(\a^\sigma)_{2})\\
+
\delta_{k,0}\cdot(-1)^{n+1} \qg_{\emptyset,1}(i_*1).
\end{equation}
\end{lm}

The following three properties are proved in~\cite[Section 4]{GST}.

\begin{lm}[Unit]\label{lm:p_unit}
For $\a_1,\ldots,\a_k\in A^*(L;R)$ and $\gamma=(\gamma_1,\ldots,\gamma_l)\in A^*(X;Q)^{\otimes l}$,
\begin{equation*} \p_{k+1,l}^{\beta}(\alpha_1,\ldots,\alpha_{i-1},1,\alpha_{i},\ldots,\alpha_k ;\gamma)=
\begin{cases}
0,&(k+1,l,\beta)\ne (1,0,\beta_0),\\
(-1)^{n+1}i_*1,&(k+1,l,\beta)= (1,0,\beta_0).
\end{cases}
		\end{equation*}
\end{lm}

\begin{lm}[Degree]\label{lm:pdeg}
For $k\ge 0$ and $\vec{\gamma}=(\gamma_1,\ldots,\gamma_l)$ with $|\gamma_j|=2$ for all $j$,
the map
\[
\pkl(\; ;\vec{\gamma}):C^{\otimes k}\lrarr \A^*(X;R)
\]
is of degree $n+1-k$.
\end{lm}

\begin{lm}[Energy zero]\label{lm:p_zero}
For any $\a=(\a_1,\ldots,\a_k)$, $\gamma=(\gamma_1,\ldots,\gamma_l)$, we have
\[
\p_{k,l}^{\beta_0}(\a;\gamma)=
\begin{cases}
0,&(k,l)\ne (1,0),\\
(-1)^{(n+1)(|\a_1|+1)}i_*\alpha_1, &(k,l)=(1,0).
\end{cases}
\]
\end{lm}

\subsubsection{Open-closed maps on a pseudo-isotopy}\label{sssec:pt}

Similarly, relabel the marked points of $\Mt_{k,l+1}(\beta)$ (see Section~\ref{sssec:mt}) to shift the indices of the evaluation maps
\begin{gather*}
\evbt_j:\Mt_{k,l+1}(\beta)\lrarr I\times L, \quad j\in\{1,\ldots,k\},\\
\evit_j:\Mt_{k,l+1}(\beta)\lrarr I\times X, \quad j\in\{0,\ldots,l\}.
\end{gather*}
For all $\beta\in\sly$, $k,l\ge 0$,  $(k,l,\beta) \not\in\{ (0,0,\beta_0)\}$, define
\[
\pt_{k,l}^{\beta}:\mC^{\otimes k}\otimes \mD^{\otimes l}\lrarr \mD
\]
by
\[
\pt_{k,l}^{\beta}(\otimes_{j=1}^k\at_j;\otimes_{j=1}^l\gt_j):= (-1)^{\varepsilon_p(\at)}(\evit_0)_*(\bigwedge_{j=1}^l\evit_j^*\gt_j \wedge\bigwedge_{j=1}^k\evbt_j^*\at_j)).
\]
Define also
\[
\pt_{0,0}^{\beta_0}:=0.
\]
Denote by
\[
\pt_{k,l}:\mC^{\otimes k}\otimes  A^*(I\times X;R)^{\otimes l}\lrarr \A^*(I\times X;R)
\]
the sum over $\beta$:
\begin{gather*}
\pt_{k,l}(\otimes_{j=1}^k\at_j;\otimes_{j=1}^l\gt_j):=
\sum_{\beta\in \sly}
T^{\beta}\pt_{k,l}^{\beta} (\otimes_{j=1}^k\at_j;\otimes_{j=1}^l\gt_j).
\end{gather*}
The deformed operators
\[
\pgt_{k,l}, \ptbg:\mC^{\otimes k}\otimes  A^*(I\times X;R)^{\otimes l}\lrarr  \A^*(I\times X;R)
\]
are defined analogously to~\eqref{eq:pg}-\eqref{eq:pbg}.

The $\pt$ operators satisfy properties analogous to those of the $\p$ operators, and in particular, we have the following structure equation.
\begin{lm}[{\cite[Proposition 6.1]{GST}}]\label{lm:ptstr}
Consider a list $\at=(\at_1,\ldots,\at_k),$ $\at_j\in \mC$, and $\gt\in (\mI_Q\mD)_2.$
Then
\[
d\pt^{\gt}_{k}(\at)
=
\sum_{\sigma\in\Z/k\Z}
(-1)^{s_\sigma^{[1]}(\at)+n+1}
\pt^{\gt}(\mgt (\at_1^\sigma) \otimes\at_{2}^\sigma)
+
\delta_{k,0}\cdot(-1)^{n+1} \qt^{\gt}_{\emptyset,1}(i_*1).
\]
\end{lm}

\subsubsection{An alternative operator}\label{sssec:alternate}

An approach alternative to the one worked out in this paper is using the operator $\q_{-1}$ defined above in Section~\ref{sssec:m}.
It is already notable that the structure equation of $\q_{-1}$ given in Lemma~\ref{lm:ceq} is reminiscent of the structure equation of $\p$ given in Lemma~\ref{lm:pstructure}.
The relation between the operators $\p_0$ and $\q_{-1}$ is made more precise via the following.

\begin{lm}\label{lm:p0q-1}
For any $\eta\in A^*(X),$ we have
$\langle \p_{0},\eta\rangle_X = (-1)^{|\eta|(n+1)}\q_{-1,1}(\eta)$.
The analogous statements hold also for the deformed operators, and on a pseudoisotopy.
\end{lm}

\begin{proof}
We prove the plain version, the other two are similar.

Denote by $\M_1$ the moduli space $\M_{0,1}(\beta)$ with the interior marked point labeled by $0$, and by $\M_2$ the same moduli space with the interior marked point labeled by $1$.
Let
$evi_0^1:\M_1\to X$ and $evi_1^2:\M_2\to X$
be the corresponding evaluation maps.
The contribution in degree $\beta\in\sly$ to $\langle \p_{0},\eta\rangle_X$ is
\begin{align*}
(-1)^{|\eta|} pt_*\big( (evi^1_0)_*1\wedge \eta \big)
=&
(-1)^{|\eta|+n|\eta|} pt_*\big( \eta\wedge (evi^1_0)_*1 \big)\\
=&
(-1)^{|\eta|+n|\eta|} pt_*(evi^1_0)_*\big( (evi_0^1)^*\eta\wedge 1 \big)\\
=&
(-1)^{|\eta|+n|\eta|} pt_*\big( (evi_1^2)^*\eta \big)\\
\end{align*}
So,
\[
\langle \p_{0},\eta\rangle_X
=(-1)^{|\eta|(n+1)} \q_{-1,1}(\eta).
\]
as desired.
\end{proof}

In view of Lemma~\ref{lm:p0q-1}, the classifying map $\Xi$ (or $\Xi_\phi$ in the real setting) can be rewritten in terms of $\qbg_{-1,1}(\theta_m)$.
We expect that the proofs formulated below via $\p_0$ can be equivalently rephrased using $\q_{-1}$. However, as mentioned in the introduction, the approach via $\p_0$ seems more natural. It allows for working directly with the underlying current, without having to apply it to forms.

To illustrate this point, in Section~\ref{sssec:surj} below, an outline the proof of Lemma~\ref{lm:ojex} in the language of $\q_{-1}$ is given in Remark~\ref{rem:ojexq}.

\subsection{Obstruction theory}\label{ssec:oj}
\subsubsection{Basic case}\label{sssec:oj}
Fix a \sababa{} multiplicative monoid $G=\{\lambda_j\}_{j=0}^\infty\subset R$ ordered as in~\eqref{eq:list}.
Let $l\ge 0.$ Suppose we have $b_{(l)}\in C$ with $\deg_Cb_{(l)}=1,$ $G(b_{(l)})\subset G,$ and
\[
\mg(e^{b_{(l)}})\equiv c_{(l)} \cdot 1\pmod{F^{E_l}C},\quad c_{(l)}\in (\mI_R)_2.
\]
Define the obstruction chains $o_j\in A^*(L)$ for $j=\kappa_l+1,\ldots,\kappa_{l+1}$ by
\[
o_j:=[\lambda_j](\mg(e^{b_{(l)}})).
\]

\begin{lm}[{\cite[Lemma 3.5]{ST2}}]\label{lm:u_even}
$|o_j|=2-\deg\lambda_j$.
\end{lm}

\begin{lm}[{\cite[Lemma 3.6]{ST2}}]\label{lm:oj_closed}
$do_i=0$.
\end{lm}

\begin{lm}[{\cite[Lemma 3.7]{ST2}}]\label{lm:oj-cj}
If $\deg\lambda_j=2,$ then $o_j=c_j\cdot 1$ for some $c_j\in\R.$ If $\deg\lambda_j\ne 2,$ then $o_j\in A^{>0}(L)$.
\end{lm}

\begin{lm}[{\cite[Lemma 3.8]{ST2}}]\label{lm:td}
If $\deg\lambda_j = 2-n$ and $(db_{(l)})_n = 0,$ then $o_j = 0$.
\end{lm}

\begin{lm}[{\cite[Lemma 3.9]{ST2}}]\label{lm:inductive}
Suppose for all $j\in\{\kappa_l+1,\ldots,\kappa_{l+1}\}$ such that $\deg\lambda_j\ne 2$, there exist $b_j\in A^{1-\deg\lambda_j}(L)$ such that $(-1)^{\deg\lambda_j}db_j=-o_j.$ Then
\[
b_{(l+1)}:=b_{(l)}+\sum_{\substack{\kappa_l+1\le j\le \kappa_{l+1}\\ \deg\lambda_j\ne 2}}\lambda_jb_j
\]
satisfies
\[
\mg(e^{b_{(l+1)}})\equiv c_{(l+1)}\cdot 1\pmod{F^{E_{l+1}}C},\qquad c_{(l+1)}\in (\mI_R)_2.
\]
\end{lm}

\begin{lm}[{\cite[Lemma 3.10]{ST2}}]\label{lm:init}
Let $\zeta\in \mI_R C$. Then
$\mg(e^\zeta)\equiv 0\pmod{F^{E_0}C}.$
\end{lm}

\subsubsection{Obstruction theory on a pseudo-isotopy}\label{sssec:ojt}

Fix a \sababa{} multiplicative monoid $G\subset R$ ordered as in~\eqref{eq:list}.
Suppose $\gt\in (\mI_Q{\mD})_2$ is closed. Let $l\ge 0,$ and suppose we have $\bt_{(l-1)}\in \mC$ such that $G(\bt_{(l-1)})\subset G$ and $\deg_{\mC}\bt_{(l-1)}=1.$
If $l\ge 1,$ assume in addition that
\begin{gather*}
\mgt(e^{\bt_{(l-1)}})\equiv \ct_{(l-1)}\cdot 1\pmod{F^{E_{l-1}}\mC}, \quad \ct_{(l-1)}\in(\mI_R\mR)_2, \quad d\ct_{(l-1)}=0.
\end{gather*}
Define the obstruction chains $\ot_j \in A^*(I\times L)$ by
\[
\ot_j:=[\lambda_j](\mgt(e^{\bt_{(l-1)}})),\qquad j=\kappa_{l-1}+1,\ldots,\kappa_l.
\]

\begin{lm}[{\cite[Lemma 3.20]{ST2}}]\label{lm:ojt_closed}
$d\ot_j=0.$
\end{lm}

\begin{lm}[{\cite[Lemma 3.21]{ST2}}]\label{lm:ut_even}
$|\ot_j|=2-\deg\lambda_j$.
\end{lm}

\begin{lm}[{\cite[Lemma 3.22]{ST2}}]\label{lm:deg_ut}
If $\deg \lambda_j = 1-n$ and $(d\bt_{(l-1)})_{n+1}=0$, then $\ot_j=0.$
\end{lm}

\begin{lm}[{\cite[Lemma 3.19]{ST2}}]\label{lm:ob1}
Suppose $\alpha \in A^1(I\times L),\,\alpha|_{\d(I\times L)} = 0,$ and $d\alpha = 0.$ Then, there exists $\eta \in A^0(I\times L),\,\eta|_{\d(I\times L)} = 0,$ such that $d \eta = \alpha + r\, dt$ for some $r\in \R.$
\end{lm}

\begin{lm}[{\cite[Lemma 3.23]{ST2}}]\label{lm:ut_relative}
Let $i=0$ or $i=1.$
Write
\[
b_{(l-1)}=j_i^*\bt_{(l-1)}, \qquad \gamma=j_i^*\gt.
\]
Suppose
\begin{gather*}
\mg(e^{b_{(l-1)}})\equiv c_{(l)}\cdot 1\pmod{F^{E_{l}}C},\quad c_{(l)}\in(\mI_R)_2.
\end{gather*}
If $\deg\lambda_j\ne 2$, then $j_i^*\ot_j=0$. If $\deg\lambda_j=2$, then $\ot_j=\ct_j\cdot 1$ with $\ct_j =  [\lambda_j](c_{(l)}).$
\end{lm}

\begin{lm}[{\cite[Lemma 3.24]{ST2}}]\label{lm:ut_exact}
Suppose for all $j\in\{\kappa_{l-1}+1,\ldots,\kappa_{l}\}$ such that $\deg\lambda_j\ne 2$, there exist $\bt_j\in A^{1-\deg\lambda_j}(I\times L)$ such that $(-1)^{\deg\lambda_j}d\bt_j=-\ot_j + \ct_j \, dt$ with $\ct_j \in \R.$ Then
\[
\bt_{(l)}:=\bt_{(l-1)}+\sum_{\substack{\kappa_{l-1}+1\le j\le\kappa_l \\ \deg\lambda_j\ne 2}}\lambda_j\bt_j
\]
satisfies
\begin{gather*}
\mgt(e^{\bt_{(l)}})\equiv \ct_{(l)}\cdot 1\pmod{F^{E_l}\mC},\quad \ct_{(l)}\in (\mI_R\mR)_2, \quad d\ct_{(l)} = 0.
\end{gather*}
\end{lm}

In addition to properties of obstruction classes, we cite here a general, topological property on isotopies which will be useful in our arguments:

\begin{lm}[{\cite[Lemma 3.15]{ST2}}]\label{lm:homotopy}
Let $M$ be a manifold with $\d M=\emptyset$ and let $\tilde{\xi}\in A^*(I\times M)$ such that $d\tilde{\xi}=0.$ Then
\[
[j_0^*\xit]=[j_1^*\xit]\in H^*(M).
\]
\end{lm}

\section{Proof of the classification theorem}\label{sec:thm1}

In this section we prove Theorem~\ref{thm:A}.


\subsection{Well-definedness of the classifying map}\label{ssec:xidef}

We need to show that the map
\[
\Xi^\flat: \{\text{separated bounding pairs}\} \lrarr (\mI_Q\Hh^*(X,L;Q))_2\oplus \bigoplus_{m=0}^r (\mI_R)_{1-n-|\theta_m|}
\]
defined by
\[
\Xi^\flat(b,\gamma) := [\gamma]\oplus\bigoplus_{m=0}^r\langle \pbg_{0}-(-1)^n\qg_{\emptyset,1}(\zeta), \theta_m\rangle_X
\]
is constant on gauge-equivalence classes.

Using notation consistent with~\cite{ST2}, set
\[
\Oh(b,\gamma):= (-1)^n\cdot \qbg_{-1,0}.
\]
The following is computed in the proof of Theorem 1 of~\cite{ST3}.
\begin{lm}\label{lm:errorinvce}
$\Oh(b',\gamma')-\Oh(b,\gamma) = pt_*i^*\qt^{\gt}_{\emptyset,0}.$
\end{lm}

\begin{lm}\label{lm:drv}
For $j\in \{0,\ldots,N\}$, we have
$\d_j\Oh = (-1)^n\qbg_{-1,1}(\gamma_j)$.
\end{lm}

\begin{proof}
This is a straightforward calculation, detailed below. The last equality follows from the separatedness assumption on $b$.
\begin{align*}
(-1)^n\d_j\Oh & = \d_j\qbg_{-1,0}\\
& = \qbg_{-1,1}(\gamma_j) + \sum_{k\ge 0}\langle \mbg_0,\d_j b\rangle_L\\
& = \qbg_{-1,1}(\gamma_j)+ c\cdot \langle 1, \d_j b\rangle_L\\
& = \qbg_{-1,1}(\gamma_j) + c \cdot \int_L \d_j b\\
& = \qbg_{-1,1}(\gamma_j).
\end{align*}
\end{proof}

\begin{lm}\label{lm:invt}
If $(\gamma,b)\sim (\gamma',b')$, then $\Xi^\flat(b,\gamma) = \Xi^\flat(b',\gamma')$.
\end{lm}

\begin{proof}
Equality in the first component, $[\gamma]=[\gamma']$, is proven in~\cite[Lemma 3.16]{ST2}.\footnote{See Section 5.5 ibid. for a detailed discussion of the different versions of relative cohomology used here.}
To consider other components, let $m\in\{0,\ldots,r\}$.
Write $\theta_m=\sum_{j}a_{mj}\gamma_j$.
By Lemma~\ref{lm:drv}, we conclude that
\[
\d_{\theta_m}\Oh = (-1)^n\qbg_{-1,1}(\theta_m).
\]
Therefore,
\[
\q^{b',\gamma'}_{-1,1}(\theta_m) - \qbg_{-1,1}(\theta_m) =
(-1)^n\d_{\theta_m}\big(\Oh(b',\gamma') - \Oh(b,\gamma)\big).
\]
Given a pseudoisotopy, we use Lemmas~\ref{lm:p0q-1} and~\ref{lm:errorinvce}:
\begin{align*}
\langle\p^{b',\gamma'}_{0},\theta_m\rangle_X - \langle \pbg_{0},\theta_m\rangle_X
&= (-1)^{|\theta_m|(n+1)}
\big(\q^{b',\gamma'}_{-1,1}(\theta_m) - \qbg_{-1,1}(\theta_m)\big)\\
&= (-1)^{|\theta|(n+1)+n}\d_{\theta_m}\big(pt_*i^*\qt^{\gt}_{\emptyset,0}\big).
\end{align*}
%
On the other hand, by Lemmas~\ref{lm:d_ll_gg},~\ref{lm:qemptcyc}, and~\ref{lm:qt_cyclic}, we have
\begin{align*}
\langle \q^{\gamma'}_{\emptyset,1}(\zeta),\theta_m\rangle_X-\langle \qg_{\emptyset,1}(\zeta),\theta_m\rangle_X
&= (-1)^{n-1+|\theta_m|+n}
\int_{I}d\ll\qt^{\gt}_{\emptyset,1}(\zetat),\thetat_m\gg_X\\
&= (-1)^{|\theta_m|+1+n(|\theta_m|+1)}
\int_{I}\ll\qt^{\gt}_{\emptyset,1}(\thetat_m),d\zetat\gg_X\\
&= (-1)^{(n+1)(|\theta_m|+1)}
\int_{I} \ll\qt^{\gt}_{\emptyset,1}(\thetat_m),i_*1\gg_X\\
&=
(-1)^{(n+1)(|\theta_m|+1)}
\int_{I} \ll i^*\qt^{\gt}_{\emptyset,1}(\thetat_m),1\gg_L\\
&=
(-1)^{|\theta|(n+1)+n+1+n}
pt_* (p^L)_* i^*\qt^{\gt}_{\emptyset,1}(\thetat_m)\\
&=
(-1)^{|\theta_m|(n+1)+1}
pt_* i^*\qt^{\gt}_{\emptyset,1}(\thetat_m)\\
\end{align*}
In total, $\Xi^\flat(b',\gamma') - \Xi^\flat(b,\gamma)=0$.
\end{proof}

\subsection{Surjectivity of the classifying map}\label{sssec:surj}

We proceed with the setup of Section~\ref{ssec:oj}. Namely, fix a monoid $G$ as in~\eqref{eq:list}, and suppose we have $b_{(l)}\in C$ with $\deg_Cb_{(l)}=1,$ $G(b_{(l)})\subset G,$ and
\begin{equation}\label{assumption}
\mg(e^{b_{(l)}})\equiv c_{(l)} \cdot 1\pmod{F^{E_l}C},\quad c_{(l)}\in (\mI_R)_2.
\end{equation}
In addition, assume
\[
\Xi_2(b_{(l)},\gamma)\equiv \vec{a}\pmod{F^{E_{l}}C}.
\]

Let $j\in \{\kappa_l+1,\ldots,\kappa_{l+1}\}$.
In Section~\ref{ssec:oj} we defined
\begin{equation}\label{eq:oj_dfn}
o_j:=[\lambda_j](\mg(e^{b_{(l)}}))
\end{equation}
and showed it is a closed form of degree $2-\deg\lambda_j$. Define now
\begin{equation}\label{eq:uclass}
\vec{\u}_{j} := [\lambda_{j}] \big(\bigoplus_{m=0}^r\langle\p^{b_{(l)},\gamma}_{0} -(-1)^n\qg_{\emptyset,1}(\zeta),\theta_m\rangle_X-\vec{a}\big)
=: \oplus_m\u_{j}^m \in \R^{\,r+1}.
\end{equation}

The proof of Theorem~\ref{thm:A} will proceed via typical obstruction theory, similarly to~\cite{ST2}, except we need to control for the additional ``obstructions'' $\u_j$ to maintain the value of $\Xi_2$ throughout the construction.
To that effect, we study $\u_j$ similarly to the study of $o_j$ as cited in Section~\ref{ssec:oj}.

\begin{lm}\label{lm:ujdeg}
If $\u_j^m\ne 0$ then $\deg \lambda_j=1-n+|\theta_m|$.
\end{lm}

\begin{proof}
By definition, $\deg a_m=1-n+|\theta_m|$.

Consider the summand $\langle \p^{b_{(l)},\gamma}_{0},\theta_m\rangle_X$. Lemma~\ref{lm:pdeg} gives $\deg \p^{b_{(l)},\gamma}_{0}=n+1$.
Thus, if $[\lambda_j](\langle\p^{b_{(l)},\gamma}_{0},\theta_m\rangle_X)\ne 0$, then
\[
\deg\lambda_j=n+1-(2n-|\theta_m|)=1-n+|\theta_m|.
\]
Consider the summand $\qg_{\emptyset,1}(\zeta)(\theta_m)$. It is well known (cited, e.g., in~\cite[Lemma 2.15]{ST3}) that the operator $\qg_{\emptyset,1}$ is of degree $2$. Therefore,
\[
\deg\qg_{\emptyset,1}(\zeta)= |\zeta|+2 = n-1+2=n+1,
\]
and
\[
\deg \langle\qg_{\emptyset,1}(\zeta),\theta_m\rangle_X=|\theta_m|+n+1-2n = |\theta_m|+1-n.
\]
\end{proof}

This being established, we need to make sure that it is possible to correct $\u_j$ whenever it is possible to correct $o_j$. Recall that $do_j=0$ and $|o_j|=2-\deg\lambda_j$, by Lemmas~\ref{lm:oj_closed} and~\ref{lm:u_even}, respectively.

Let
$\bar{b}_m\in A^{n-|\theta_m|}(L)$ be representatives of the basis of $H^{>0}(L)$ dual to $\{i^*[\theta_m]\}_m\subset H^{<n}(L)$, namely,
\begin{equation}\label{eq:bbar}
d\bar{b}_m=0\quad \text{ and } \quad \langle \bar{b}_m \, , i^*\theta_\ell\rangle_L =(-1)^n\delta_{m,\ell},
\end{equation}
where $\delta_{m,\ell}$ is the Kronecker delta.

\begin{lm}\label{lm:corrou}
If $[o_{j}]=0\in H^{2-\deg \lambda_j}(L)$, then there exists $b_{j}\in A^{1-\deg\lambda_j}(L)$ such that
\[
(-1)^{\deg \lambda_j}db_j=-o_j\quad\text{and}\quad
\langle b_j,i^*\theta_m\rangle_L=(-1)^{n+1}\u_j^m, \:\forall m.
\]
\end{lm}

\begin{proof}
Take $b_{j}^\prime\in A^{1-\deg\lambda_j}(L)$ is such that $db_{j}^\prime = (-1)^{1+\deg\lambda_j}o_{j}$.
Set $f_{j}^m:=\langle b_{j}^\prime,i^*\theta_m\rangle_L$,
and take $b_{j}:=b_{j}^\prime + \sum_{m=0}^r (-f_{j}^m+(-1)^{n+1}\upsilon^m_{j})\bar{b}_m$, with $\bar{b}_m$ as in~\eqref{eq:bbar}.
Then by Lemma~\ref{lm:ujdeg} we have
$|b_j|=1-\deg\lambda_j$, by definition we have
$db_j=db^\prime_j=(-1)^{1+\deg\lambda_j}o_j$, and
\[
\langle b_{j},i^*\theta_m\rangle_L
= \langle b^\prime_j,i^*\theta_m\rangle_L -f_j^m\langle\bar{b}_m,i^*\theta_m\rangle_L
+(-1)^{n+1}\u_j^m\langle \bar{b}_m,i^*\theta_m\rangle_L
=(-1)^{n+1}\u_j^m.
\]
\end{proof}

\begin{lm}\label{lm:ximod}
Suppose for all $j\in\{\kappa_l+1,\ldots,\kappa_l\}$ such that $\deg \lambda_j\ne 2$, there exist $b_j\in A^{1-\deg \lambda_j}(L)$ such that
\[
(-1)^{\deg \lambda_j}db_j=-o_j\quad\text{and}\quad
\langle b_j,i^*\theta_m\rangle_L=(-1)^{n+1}\u^m_j, \:\forall m.
\]
Then
$b_{(l+1)}:=b_{(l)}+\sum_{\substack{\kappa_l+1\le j\le \kappa_{l+1}\\ \deg\lambda_j\ne 2}}\lambda_jb_j$
satisfies
\[
\Xi_2(b_{(l+1)},\gamma)\equiv \vec{a}\pmod{F^{E_{l+1}}C}
\]
and
\[
\mg(e^{b_{(l+1)}})\equiv c_{(l+1)}\cdot 1\pmod{F^{E_{l+1}}C},\qquad c_{(l+1)}\in (\mI_R)_2.
\]
\end{lm}

\begin{proof}
Write $a_m^{(l)}:= \sum_{j\le \kappa_l}\lambda_j[\lambda_j](a_m)$ and
$a_m^{(l+1)}:= \sum_{j\le \kappa_{l+1}}\lambda_j[\lambda_j](a_m)$. In particular, $a_m^{(l)}\equiv a_m \pmod{F^{E_l}R}$ and
\[
\langle \p_{0}^{b_{(l)},\gamma}-(-1)^n\qg_{\emptyset,1}(\zeta), \theta_m\rangle_X \equiv a_m^{(l+1)}+\sum_{\kappa_l+1\le j\le \kappa_{l+1}}\lambda_j \u_j^m \pmod{F^{E_{l+1}}R}.
\]

For the first identity, compute the $m$-th component:
\begin{align*}
\big(\Xi_2(b_{(l+1)},\gamma)\big)_m
&=
\langle\p^{b_{(l+1)},\gamma}_{0} -(-1)^n\qg_{\emptyset,1}(\zeta),\theta_m\rangle_X\\
&\equiv
\langle\p^{b_{(l)},\gamma}_{0}-(-1)^n\qg_{\emptyset,1}(\zeta),\theta_m\rangle_X
+\sum_{\substack{\kappa_l+1\le j\le \kappa_{l+1}\\ \deg \lambda_j\ne 2}}
\langle\p_{1}^{\beta_0}(\lambda_{j}b_{j}), \theta_m\rangle_X\\
&\equiv
a_m^{(l+1)}
+ \sum_{\kappa_l+1\le j\le \kappa_{l+1}}\lambda_j \u_j^m
+ \sum_{\substack{\kappa_l+1\le j\le \kappa_{l+1}\\ \deg \lambda_j\ne 2}}\langle i_*(\lambda_{j}b_{j}),\theta_m\rangle_X.\\
\intertext{Note that if $\deg\lambda_j=2$ then Lemma~\ref{lm:ujdeg} gives $\u_j^m=0$ for all $m$. So, we can remove these summands from the second expression. Further, use Lemma~\ref{lm:pairing} to rewrite the third expression, and get}
&=a_m^{(l+1)}
+ \sum_{\substack{\kappa_l+1\le j\le \kappa_{l+1}\\ \deg \lambda_j\ne 2}}
\big(\lambda_j\u_{j}^m + (-1)^n\langle\lambda_{j}b_{j},i^*\theta_m\rangle_L\big).\\
\shortintertext{The last sum vanishes by the assumption on $b_j$, and we have}
&= a_m^{(l+1)}\\
&\equiv  a_m
\pmod{F^{E_{l+1}}R}.
\end{align*}

For the second identity, use Lemma~\ref{lm:inductive}.

\end{proof}

We are almost ready to prove surjectivity of $\Xi_2$. The only preliminary left is verifying that, under the assumptions of Theorem~\ref{thm:A}, obstruction classes are indeed exact.
\begin{lm}\label{lm:ojex}
Assume $i^*:H^{n-|o_j|}(X;\R)\to H^{n-|o_j|}(L;\R)$ is surjective,
and assume $[L]=0\in H_n(X;\R)$.
Then $[o_j]=0\in H^*(L;\R)$.
\end{lm}

\begin{proof}

Fix any $a\in H^j(L)$ with $j=n-|o_j|$. By assumption, $i^*: H^j:(X)\to H^j(L)$ is onto, so there exists $\eta\in A^*(X)$ such that $d\eta =0$ and $i^*[\eta] = a$. Then by Lemma~\ref{lm:pairing} we have
\[
\langle [o_j],a\rangle_L=\langle o_j,i^*\eta\rangle_L
=(-1)^{n|\eta|}\langle i_*o_j,\eta\rangle_X.
\]

Use the structure equation from Lemma~\ref{lm:pstructure} and Lemma~\ref{lm:p_zero} to get
\[
0\equiv
\sum_{\substack{\kappa_l+1\le j\le \kappa_{l+1}\\ \deg\lambda_j\ne 2}}\lambda_ji_*o_j
+c_{(l+1)}\cdot i_*1
-d\p^{b_{(l)},\gamma}_{0}
+(-1)^{n+1}\qg_{\emptyset,1}(i_*1) \pmod{F^{E_{l+1}}C}.
\]
Pairing the equation with $\eta$ and using Lemma~\ref{lm:exact}
we get
\begin{align*}
0=&\lambda_j\langle i_*o_j,\eta\rangle_X +c_{(l+1)}\langle i_*1,\eta\rangle_X -\langle d\p^{b_{(l)},\gamma}_{0},\eta\rangle_X +(-1)^{n+1}\langle\qg_{\emptyset,1}(i_*1),\eta\rangle_X\\
=&\lambda_j\langle i_*o_j,\eta\rangle_X
+ \langle d\big(c_{(l+1)}\zeta-\p_0^{b_{(l)},\gamma} +(-1)^{n+1}\qg_{\emptyset,1}(\zeta)\big),\eta\rangle_X\\
=&\lambda_j\langle i_*o_j,\eta\rangle_X.
\end{align*}
Therefore,
\[
\langle o_j,a\rangle_L=(-1)^{n|\eta|}\langle i_*o_j,\eta\rangle_X=0.
\]
Since $a$ was arbitrary, it follows that $[o_j]=0\in H^*(L)$.
\end{proof}

\begin{rem}\label{rem:nosurj}
Without the surjectivity assumption on $i^*$, the argument above shows that $[i_*o_j]=0\in H^*(X)$.
\end{rem}

\begin{rem}\label{rem:ojexq}
In view of Lemma~\ref{lm:p0q-1}, an equivalent proof can be formulated using the structure equation of $\q_{-1}$. As promised in Section~\ref{sssec:alternate}, we outline here the argument.
It makes use of the more general $\q$ operators, as defined in~\cite{ST1}, by which we define $\q^{b_{(l)},\gamma}_{0,1}$ as a deformation of
\[
\q_{0,1}^\beta(\eta)=-(evb_0)_*(evi_1)^*\eta.
\]
Then, the structure equation for the $\q_{-1}$ operators, given in~\cite[Proposition 2.5]{ST1}, gives
\begin{align*}
0&= \langle \q_{0,1}^{b_{(l)},\gamma}(\eta),\m^{b_{(l)},\gamma}_0\rangle_L +\i(\qg_{\emptyset,1}(\eta))\\
& \equiv \langle \q_{0,1}^{b_{(l)},\gamma}(\eta),c_{(l)}\cdot 1\rangle_L
\pm \sum_{\kappa_l\le j\le \kappa_l+1}\lambda_j\cdot\langle \q_{0,1}^{\beta_0}(\eta),o_j\rangle_L
+\i(\qg_{\emptyset,1}(\eta))\\
\shortintertext{by top degree and energy zero axioms,}
& = c_{(l)}\cdot\langle i^*\eta,1\rangle_L
\pm \sum_{\kappa_l\le j\le \kappa_l+1}\lambda_j\cdot \langle i^*\eta,o_j\rangle_L
+\i(\qg_{\emptyset,1}(\eta))\\
\shortintertext{which, by Lemma~\ref{lm:ceq}, is}
& \equiv \pm \sum_{\kappa_l\le j\le \kappa_l+1}\lambda_j\cdot \langle i^*\eta,o_j\rangle_L
\pmod{F^{E_{l+1}}C}.
\end{align*}
Using the surjectivity of $i^*$, as in the original version of the proof, we get that $[o_j]=0$. As in Remark~\ref{rem:nosurj}, without a surjectivity assumption this argument too shows that $[i_*o_j]=0$.
\end{rem}

We now put the pieces together to obtain surjectivity:
\begin{prop}\label{prop:surj}
Assume $i^*:H^j(X;\R)\to H^j(L;\R)$ is surjective for all $j<n$ and let $\theta_m\in A^*(X)$ be closed homogeneous forms so that $i^*[\theta_m]$ generate $H^{<n}(L;\R)$. Assume in addition that $[L]=0\in H_n(X,\R)$.
Define the map $\Xi$ as in Section~\ref{sssec:thmA}.
Then for any closed $\gamma\in (\mI_QD)_2$ and any $\vec{a}=\{a_m\}\in \oplus_{m=0}^r(\mI_R)_{1-n+|\theta_m|}$, there exists a bounding chain $b$ for $\mg$ so that $\Xi_2(\gamma,b)=\vec{a}$.
\end{prop}

\begin{proof}
Fix $\vec{a}=\oplus_{m=0}^ra_m$ with $a_m \in (\mI_R)_{|\theta_m|+1-n}$ and $\gamma \in (\mI_QD)_2$. Write $G(\vec{a})$ in the form of a list as in~\eqref{eq:list}.

Recall the definition of the forms $\bar{b}_m$ from~\eqref{eq:bbar}.
Let
\[
b_{(0)} = \sum_{m=0}^r a_m\bar{b}_m.
\]
By Lemma~\ref{lm:init}, the chain $b_{(0)}$ satisfies
\[
\mg(e^{b_{(0)}})\equiv 0=c_{(0)}\cdot 1\pmod{F^{E_0}C},\qquad c_{(0)}=0.
\]
Moreover, $db_{(0)}=0$, $\deg_{C}b_{(0)}=1-n+|\theta_m|+(n-|\theta_m|)=1$, and, using $\nu(b_{(0)}),\nu(\gamma)>0$ with Lemma~\ref{lm:czero},
\[
\Xi_2(b_{(0)},\gamma) =\oplus_m\langle \p^{b_{(0)},\gamma}_{0},\theta_m\rangle_X \equiv \oplus_m 0
\equiv \oplus_m a_m
\pmod{F^{E_0}C}.
\]

Proceed by induction. Suppose we have $b_{(l)}\in C$ with $\deg_Cb_{(l)}=1$, $G(b_{(l)})\subset G(\vec{a}),$ and
\begin{gather*}
(db_{(l)})_n =0,\qquad\Xi_2(b_{(l)},\gamma)\equiv\vec{a} \pmod{F^{E_l}C},\\
\mg(e^{b_{(l)}})\equiv c_{(l)}\cdot 1\pmod{F^{E_l}C},\quad c_{(l)}\in (\mI_R)_2.
\end{gather*}
Define the obstruction chains $o_j$ by~\eqref{eq:oj_dfn} and the elements $\u_j^m$ by~\eqref{eq:uclass}.
By Lemma~\ref{lm:oj_closed}, we have $do_j=0$, and by Lemma~\ref{lm:u_even} we have $o_j\in A^{2-\deg\lambda_j}(L).$ To apply Lemma~\ref{lm:ximod}, it is necessary to choose forms $b_j\in A^{1-\deg\lambda_j}(L)$ such that $(-1)^{\deg\lambda_j}db_j=-o_j$
and $\langle b_j,i^*\theta_m\rangle_L=-\u_j^m$,
for all $j\in\{\kappa_l+1,\ldots,\kappa_{l+1}\}$ with $\deg\lambda_j\ne 2$.

If $\deg \lambda_j = 2-n,$ Lemma~\ref{lm:td} gives $o_j = 0,$ so we choose $b_j = -\u_j^m\bar{b}_m$.

If $2-n < \deg \lambda_j<~2,$ then $0 < |o_j| < n.$
Then Lemma~\ref{lm:ojex} gives $[o_j] = 0 \in H^*(L;\R)$, and we choose $b_j$ such that $(-1)^{\deg\lambda_j}db_j = -o_j.$
By Lemma~\ref{lm:corrou} we may choose $b_j$ so that $\langle b_j,i^*\theta_m\rangle_L=-\u_j^m$ for all $m$.

For other possible values of $\deg \lambda_j,$ degree considerations imply $o_j = 0,$ so we choose $b_j =-\u_j^m\bar{b}_m$.

Lemma~\ref{lm:ximod} now guarantees that $b_{(l+1)}:=b_{(l)}+\sum_{\substack{\kappa_l+1\le j\le \kappa_{l+1}\\ \deg\lambda_j\ne 2}}\lambda_jb_j$
satisfies
\[
\Xi_2(b_{(l+1)},\gamma)\equiv \vec{a}\pmod{F^{E_{l+1}}C}
\]
and
\[
\mg(e^{b_{(l+1)}})\equiv c_{(l+1)}\cdot 1\pmod{F^{E_{l+1}}C},\qquad c_{(l+1)}\in (\mI_R)_2.
\]
Since $db_j = 0$ when $\deg \lambda_j \in \{2-n,1-n\}$, it follows that $(db_{(l+1)})_n = (db_{(l)})_n=0$.

Thus, the inductive process gives rise to a convergent sequence $\{b_{(l)}\}_{l=0}^\infty$ where $b_{(l)}$ is bounding modulo $F^{E_l}C$.
Taking the limit as $l$ goes to infinity, we obtain
\[
b=\lim_l b_{(l)},\quad\deg_C b=1, \quad \Xi_2(b,\gamma) = \vec{a}, \quad \mg(e^b)= c\cdot 1,\quad c=\lim_lc_{(l)}\in (\mI_R)_2.
\]
\end{proof}

\subsection{Injectivity of the classifying map}\label{ssec:inj}

Let $(\gamma,b)$ and $(\gamma',b')$ be bounding pairs.

\begin{lm}\label{lm:preexact}
Suppose $\Xi_2(\gamma,b)=\Xi_2(\gamma',b')$. Then for all $\xi\in Span([\theta_0]
,\ldots,[\theta_r])\subset H^{<n}(X;\R)$ we have
$\langle\p_{0}^{b,\gamma}-(-1)^n\qg_{\emptyset,1}(\zeta),\xi\rangle_X =\langle\p_{0}^{b',\gamma'}-(-1)^n\q^{\gamma'}_{\emptyset,1}(\zeta),\xi\rangle_X$.
\end{lm}
\begin{proof}
By assumption, we can decompose $\xi=\sum_{m=0}^r \xi_m\cdot [\theta_m]$ with $\xi_m\in \R$. The result follows from linearity of currents.
\end{proof}

From now on, we assume that $\Xi(\gamma,b)=\Xi(\gamma',b')$.
We will build, by induction on the filtration from Section~\ref{ssec:filtration}, a pair $(\gt,\bt)$ that realizes a gauge equivalence between $(\gamma,b)$ and $(\gamma',b')$. As a first step, we establish properties of the obstruction classes that will subsequently enable us to carry out the inductive step. Some properties are completely general and were cited in Section~\ref{sssec:ojt}. Below we work out additional properties that are specific to our topological assumptions.

Proceed with the setup of Section~\ref{ssec:oj}. Namely, fix a monoid $G$ as in~\eqref{eq:list} such that $G(b),G(b')\subset G$.
Suppose $\gt\in (\mI_Q{\mD})_2$ is closed. Let $l\ge 0,$ and suppose we have $\bt_{(l-1)}\in \mC$ such that $G(\bt_{(l-1)})\subset G$ and $\deg_{\mC}\bt_{(l-1)}=1.$
If $l\ge 1,$ assume in addition that
\begin{gather*}
\mgt(e^{\bt_{(l-1)}})\equiv \ct_{(l-1)}\cdot 1\pmod{F^{E_{l-1}}\mC}, \quad \ct_{(l-1)}\in(\mI_R\mR)_2, \quad d\ct_{(l-1)}=0.
\end{gather*}
For $j\in \{\kappa_{l-1}+1,\ldots,\kappa_l\}$, define the obstruction chains $\ot_j \in A^*(I\times L)$ by
\begin{equation}\label{eq:ojt_dfn}
\ot_j:=[\lambda_j](\mgt(e^{\bt_{(l-1)}})).
\end{equation}

\begin{lm}\label{lm:otjex}
For all $j\in \{\kappa_{l-1}+1,\ldots,\kappa_l\}$ such that $1<|\ot_j|<n+1$, we have
$[\ot_j]=0\in H^*(I\times L,\d(I\times L);\R)$.
\end{lm}
\begin{proof}
By assumptions on $\bt_{(l-1)}$ and by Lemma~\ref{lm:ut_relative}, we can write
\[
\mt^{\bt_{(l-1)},\gt}_{0}
\equiv \ct_{(l-1)}\cdot 1
+\sum_{\substack{\kappa_{l-1}+1\le j\le \kappa_{l}\\ \deg\lambda_j = 2}}\lambda_j \ct_j\cdot 1
+\sum_{\substack{\kappa_{l-1}+1\le j\le \kappa_{l}\\ \deg\lambda_j\ne 2}}\lambda_j \ot_j \pmod{F^{E_{l}}C}
\]
with $\ct_j=c_j\in \R$.
Set $\chi_l:=\ct_{(l-1)}+\sum_{\substack{\kappa_{l-1}+1\le j\le \kappa_{l}\\ \deg\lambda_j = 2}}\lambda_j \ct_j$.
Then by Lemmas~\ref{lm:ptstr},~\ref{lm:p_unit}, and~\ref{lm:p_zero}, we get
\begin{align}\label{eq:ojt_ex}
0
& =
\pt^{\bt_{(l-1)},\gt}_{1}(\mt^{\bt_{(l-1)},\gt}_{0})
+
d\pt_{0}^{\bt_{(l-1)},\gt}
+(-1)^{n+1} \qt_{\emptyset,1}^{\gt}(i_*1)\\
& \equiv \pt^{\bt_{(l-1)},\gt}_{1}(\chi_l\cdot 1)
+ \sum_{\substack{\kappa_{l-1}+1\le j\le \kappa_{l}\\ \deg\lambda_j\ne 2}}\pt^{\beta_0}_{1}(\lambda_j\ot_j)
+ d\pt_{0}^{\bt_{(l-1)},\gt}
+(-1)^{n+1} \qt_{\emptyset,1}^{\gt}(i_*1)\notag\\
& \equiv
-\chi_l\cdot i_*1
-\sum_{\substack{\kappa_{l-1}+1\le j\le \kappa_{l}\\ \deg\lambda_j\ne 2}}\lambda_ji_*\ot_j
+
d\pt_{0}^{\bt_{(l-1)},\gt}
+
(-1)^{n+1} \qt_{\emptyset,1}^{\gt}(i_*1)
\pmod{F^{E_{l}}C}.\notag
\end{align}

Fix $j_0\in \{\kappa_{l-1}+1,\ldots,\kappa_l\}$ such that $1<|\ot_{j_0}|<n+1$, or equivalently, $1-n<\deg\lambda_{j_0}<1$. In particular, $\deg\lambda_{j_0}\ne 2$.
We would like to verify that an arbitrary closed, homogeneous $\xit\in A^*(I\times L)$ satisfies $\langle \ot_{j_0},\xit\rangle_{I\times L}=0$.
The value $\langle\ot_{j_0},\xit\rangle_{I\times L}$ vanishes unless
$|\xit|=n+1-|\ot_{j_0}|$. Using the assumption $|\ot_{j_0}|>1$ we may therefore take $|\xit|<n$.

Since $H^*(I\times M)\simeq H^*(M)$ for $M=L,X,$ there exists $\etat\in Span(\theta_0,\ldots,\theta_r)$ such that $i^*[\etat]=[\xit]\in H^*(I\times L)$.
Then pairing equation~\eqref{eq:ojt_ex} above with $\xit$ yields
\begin{align*}
\langle \ot_j,\xit\rangle_{I\times L} &
= \langle \ot_j, i^*\etat\rangle_{I\times L}\\
& = (-1)^{n|\etat|}\int_I \ll i_*\ot_j,\etat\gg_X\\
& = (-1)^{n|\etat|}[\lambda_j]\big(\int_I \ll \chi_l\cdot i_*1 + d \pt^{\bt_{(l-1)},\gt}_{0} + (-1)^{n+1}\qt_{\emptyset,1}^{\gt}(i_*1), \etat\gg_X\big).
\end{align*}
To compute the right-hand side note that, by Lemma~\ref{lm:exact}, there exists $\zetat$ such that $i_*1=d\zetat$.
Note also that $j_0^*\chi_l=c_{(l)}$ and $j_1^*\chi_l=c'_{(l)}$ and recall that, by Lemma~\ref{lm:cinvt}, we have $c=c'$.
So,
\begin{align*}
\int_I \ll &\chi_l\cdot i_*1 - d \pt^{\bt_{(l-1)},\gt}_{0} - (-1)^{n+1}\qt_{\emptyset,1}^{\gt}, \etat\gg_X=\\
& = \int_I  d\ll \chi_l\zetat - \pt^{\bt_{(l-1)},\gt}_{0}- (-1)^{n+1}\qt_{\emptyset,1}^{\gt}(\zetat), \etat\gg_X\\
& \equiv \pm \langle c'\zeta - \p^{b',\gamma'}_{0}-(-1)^{n+1}\q^{\gamma'}_{\emptyset,1}(\zeta), j_1^*\etat\rangle_X \\
&\hspace{15em} \mp
\langle c\zeta - \p^{b,\gamma}_{0}-(-1)^{n+1}\q^{\gamma}_{\emptyset,1}(\zeta), j_0^*\etat\rangle_X \pmod{F^{E_l}R}.\\
\shortintertext{Since $|j_i^*\etat|=|\xit|<n$, we can use Lemma~\ref{lm:preexact} to rewrite the last expression as}
& = \pm \langle c\zeta - \p^{b,\gamma}_{0}-(-1)^{n+1}\q^{\gamma}_{\emptyset,1}(\zeta), (j_1^*-j_0)^*\etat\rangle_X .\\
\shortintertext{By Lemma~\ref{lm:homotopy} there exists $\hat\eta$ so that this equals }
& = \pm \langle c\zeta - \p^{b,\gamma}_{0}-(-1)^{n+1}\q^{\gamma}_{\emptyset,1}(\zeta), d\hat\eta\rangle_X\\
& = \pm \langle c\cdot i_*1 - d\p^{b,\gamma}_{0}-(-1)^{n+1}\q^{\gamma}_{\emptyset,1}(i_*1), \hat\eta\rangle_X\\
\shortintertext{and by Lemma~\ref{lm:ptstr},}
& = \pm \langle 0,\hat\eta\rangle\\
&=0.
\end{align*}
We thus see that all closed $\xit\in A^*(I\times L)$ satisfy
\[
\langle \ot_j,\xit\rangle_{I\times L}
=\pm \langle [\lambda_j]\big(c\cdot i_*1 - d\p^{b,\gamma}_{0}-(-1)^{n+1}\q^{\gamma}_{\emptyset,1}(i_*1)\big), \hat\eta\rangle_X
=0.
\]
Furthermore, Lemma~\ref{lm:ut_relative} gives $j_0^*\ot_j=j_1^*\ot_j=0$.
We conclude by Poincar\'e-Lefschetz duality that
\[
[\ot_j]=0\in H^*(I\times L, \d(I\times L);\R).
\]

\end{proof}

\begin{prop}\label{prop:inj}
Assume $i^*:H^j(X;\R)\to H^j(L;\R)$ is surjective for all $j<n$ and let $\theta_m\in A^*(X)$ be closed homogeneous forms so that $i^*[\theta_m]$ generate $H^{<n}(L;\R)$. Assume in addition that $[L]=0\in H_n(X,\R)$.
Let $(\gamma,b)$, $(\gamma',b')$, be bounding pairs such that $\Xi(\gamma,b)=\Xi(\gamma',b')$. Then $(\gamma,b)\sim (\gamma',b')$.
\end{prop}

\begin{proof}
We construct a pseudoisotopy $(\mgt,\bt)$ from $(\mg,b)$ to $(\m^{\gamma'},b').$
The construction here is exactly parallel to the cases covered in~\cite{ST2}, but we repeat it for the sake of completeness.

Let $\{J_t\}_{t \in I}$ be a path from $J$ to $J'$ in $\J$.
Let $\xi\in (\mI_QD)_1$
be such that
\[
\gamma'-\gamma=d\xi
\]
and define
\[
\gt:=\gamma+t(\gamma'-\gamma)+dt\wedge\xi\in (\mI_Q\mD)_2.
\]
Then
\begin{gather*}
d\gt=dt\wedge\d_t(\gamma+t(\gamma'-\gamma))-dt\wedge d\xi=0,\\
j_0^*\gt=\gamma,\quad j_1^*\gt=\gamma',\quad i^*\gt=0.
\end{gather*}

We now move to constructing $\bt$. Write $G(b,b')$ in the form of a list as in~\eqref{eq:list}.
Since $\int_Lb'=\int_Lb,$ there exists $\eta\in \mI_RC$ such that $G(\eta)\subset G(b,b')$, $|\eta|=n-1$, and $d\eta=(b')_n-(b)_n.$
Write
\[
\bt_{(-1)}:=b+t(b'-b)+dt\wedge\eta\in\mC.
\]
Then
\begin{gather*}
(d\bt_{(-1)})_{n+1}=dt\wedge(b'-b)_n-dt\wedge d\eta=0,\\
j_0^*\bt_{(-1)}=b,\qquad j_1^*\bt_{(-1)}=b'.
\end{gather*}
Let $l\ge 0.$ Assume by induction that we have constructed $\bt_{(l-1)}\in \mC$ with $\deg_{\mC}\bt_{(l-1)}=1,$ $G(\bt_{(l-1)})\subset G(b,b'),$ such that
\[
(d\bt_{(l-1)})_{n+1}=0,\qquad
j_0^*\bt_{(l-1)}=b,\qquad j_1^*\bt_{(l-1)}=b',
\]
and if $l\ge 1$, then
\[
\mgt(e^{\bt_{(l-1)}})\equiv \ct_{(l-1)}\cdot 1\pmod{F^{E_{l-1}}\mC}, \qquad \ct_{(l-1)}\in(\mI_R\mR)_2, \quad d\ct_{(l-1)} = 0.
\]
Define the obstruction chains $\ot_j$ by~\eqref{eq:ojt_dfn}.
By Lemma~\ref{lm:ut_even} we have $\ot_j\in A^{2-\deg\lambda_j}(I\times L),$ by Lemma~\ref{lm:ojt_closed} we have $d\ot_j=0$, and by Lemma~\ref{lm:ut_relative} we have $\ot_j|_{\d(I\times L)}=0$ whenever $\deg \lambda_j \ne 2$.
To apply Lemma~\ref{lm:ut_exact}, it is necessary to choose forms $\bt_j\in A^{1-\deg\lambda_j}(I\times L)$ such that $(-1)^{\deg\lambda_j}d\bt_j=-\ot_j+ \ct_j \,dt,\, \ct_j \in \R,$ for $j\in\{\kappa_l+1,\ldots,\kappa_{l+1}\}$ such that $\deg\lambda_j\ne 2.$

If $\deg \lambda_j = 1-n,$ Lemma~\ref{lm:deg_ut} implies that $\ot_j = 0,$ so we choose $\bt_j = 0.$

If $\deg \lambda_j = 1,$ Lemma~\ref{lm:ob1} gives $\bt_j$ such that $(-1)^{\deg\lambda_j}d \bt_j = -\ot_j + \ct_j\, dt,\,\ct_j \in \R,$ and $\bt_j|_{\d(I\times L)} = 0.$

If $1-n < \deg \lambda_j< 1,$ then $1 < |\ot_j| < n+1.$ So, Lemma~\ref{lm:otjex} implies that $[\ot_j] = 0 \in H^*(I\times L,\d(I\times L);\R).$
Thus, we choose $\bt_j$ such that $(-1)^{\deg\lambda_j}d\bt_j = -\ot_j$ and $\bt_j|_{\d(I\times L)} = 0.$

For other possible values of $\deg \lambda_j,$ degree considerations imply $\ot_j = 0,$ so we choose $\bt_j =0.$

By Lemma~\ref{lm:ut_exact}, the form
\[
\bt_{(l)}:=\bt_{(l-1)}+\sum_{\substack{\kappa_{l-1}+1\le j\le\kappa_l \\ \deg\lambda_j\ne 2}}\lambda_j\bt_j
\]
satisfies
\[
\mgt(e^{\bt_{(l)}})\equiv \ct_{(l)}\cdot 1\pmod{F^{E_{l}}\mC}, \qquad \ct_{(l)}\in(\mI_R\mR)_2, \qquad d\ct_{(l)} = 0.
\]
Since we have chosen $\bt_j =0$ when $\deg \lambda_j = 1-n,$ it follows that
\[
(d\bt_{(l)})_{n+1}= (d\bt_{(l-1)})_{n+1}=0.
\]
Since $\bt_j|_{\d(I\times L)}=0$ for all $j\in\{\kappa_l+1,\ldots,\kappa_{l+1}\}$ such that $\deg\lambda_j\ne 2$, we have
\[
j_0^*\bt_{(l)}=j_0^*\bt_{(l-1)}=b,\qquad j_1^*\bt_{(l)}=j_1^*\bt_{(l-1)}
= b'.
\]

Taking the limit $\bt=\lim_l\bt_{(l)},$ we obtain a bounding chain for $\mgt$ that satisfies $j_0^*\bt=b$ and $j_1^*\bt=b'$. So, $(\mgt,\bt)$ is a pseudoisotopy from $(\mg,b)$ to $(\m^{\gamma'},b').$

\end{proof}

\subsection{Proof of Theorem~\ref{thm:A}}

Surjectivity of $\Xi$ follows from Proposition~\ref{prop:surj}, and injectivity from Proposition~\ref{prop:inj}.

\section{The real setting}\label{sec:real}

In this section we prove Theorem~\ref{thm:B}.
Recall our conventions and notation from Section~\ref{sssec:real}. In particular, we assume throughout that $\s$ is induced by a spin structure and $L\subset \fix(\phi)$.

\subsection{Preliminaries}

Recall the definition of real elements~\eqref{eq:relt}. In the next two sections, we will use the following properties of real elements.

\begin{lm}[{\cite[Lemma 4.6]{ST2}}]\label{lm:reob}
Suppose $\gamma \in \mathcal{I}_Q D,\,b \in C,$ are real with $\deg b = 1.$ Then $\mg_k(b^{\otimes k})$ is real.
\end{lm}

\begin{lm}[{\cite[Lemma 4.8]{ST2}}]\label{lm:reR}
Suppose $n \not \equiv 1 \pmod 4.$ An element $a \in R$ is real if and only if $\deg a \equiv 2$ or $1-n \pmod 4.$
\end{lm}

\begin{lm}[{\cite[Lemma 4.14]{ST2}}]\label{lm:reobt}
Suppose $\gt \in \mathcal{I}_Q \mD,\,\bt \in \mC,$ are real with $\deg \bt = 1.$ Then $\mgt_k(\bt^{\otimes k})$ is real.
\end{lm}

\subsection{Proof of Theorem \ref{thm:B}}

By verbatim the same argument as Lemma~\ref{lm:invt}, the map $\Xi_\phi$ is well-defined. It remains to prove that $\Xi_\phi$ is a bijection.

\subsubsection{Surjectivity of the classifying map}

We proceed with the setup of Section~\ref{ssec:oj}. Namely, fix a monoid $G$ as in~\eqref{eq:list}, and suppose we have a real element $b_{(l)}\in C$ with $\deg_Cb_{(l)}=1,$ $G(b_{(l)})\subset G,$ and
\begin{equation*}
\mg(e^{b_{(l)}})\equiv c_{(l)} \cdot 1\pmod{F^{E_l}C},\quad c_{(l)}\in (\mI_R)_2.
\end{equation*}
In addition, assume
\[
\Xi_{\phi,2}(b_{(l)},\gamma)\equiv \vec{a}\pmod{F^{E_{l}}C}.
\]

Let $j\in \{\kappa_l+1,\ldots,\kappa_{l+1}\}$.
We proceed with the definitions~\eqref{eq:oj_dfn} and~\eqref{eq:uclass}, namely,
\begin{equation*}
o_j:=[\lambda_j](\mg(e^{b_{(l)}}))
,\qquad
\vec{\u}_{j} := [\lambda_{j}] \big(\bigoplus_{m=0}^r\langle\p^{b_{(l)},\gamma}_{0} -(-1)^n\qg_{\emptyset,1}(\zeta),\theta_m\rangle_X-\vec{a}\big)
=: \oplus_m\u_{j}^m.
\end{equation*}
\begin{rem}\label{rem:ureal}
By Lemmas~\ref{lm:ujdeg} and~\ref{lm:reR}, if there exists an $m$ such that $\u_j^m\ne 0$ then $\lambda_j$ is real, for all $j$. In other words, $\lambda_j\u_j^m$ is real for all $j,m$.
\end{rem}
Properties of $o_j$ and $\u_j^m$ proved in Sections~\ref{sssec:oj} and~\ref{sssec:surj}, such as degree and exactness, hold in the real setting as well.

Analogously to~\eqref{eq:bbar}, let
$\bar{b}_m\in A^{n-|\theta_m|}(L)$ be representatives of the basis of $\oplus_{j\in \mQ(n)}H^{n-j}(L)$ dual to $\{i^*[\theta_m]\}_m\subset \oplus_{j\in \mQ(n)}H^{j}(L)$, namely,
\begin{equation}\label{eq:rebbar}
d\bar{b}_m=0\quad \text{ and } \quad \langle \bar{b}_m \, , i^*\theta_\ell\rangle_L =(-1)^n\delta_{m,\ell}\;.
\end{equation}
Note that Lemma~\ref{lm:corrou} is still true when $m$ is running over the smaller set $\mQ(n)$, by verbatim the same argument.

\begin{prop}\label{prop:spin_b_exist}
Suppose $n \not \equiv 1 \pmod 4,$ and $i^*:H^j:(X;\R)\to H^j(L;\R)$ is surjective for all $j \equiv 3, n \pmod 4$. Then for any real closed $\gamma \in (\mI_QD)_2$ and any $\vec{a} \in \oplus_{m=0}^r(\mI_R)_{|\theta_m|+1-n}$, there exists a real bounding chain $b$ for $\mg$ such that
$
\Xi_{\phi,2}(b,\gamma)=\vec{a}.
$
\end{prop}
\begin{proof}
Fix $\vec{a}=\oplus_{m=0}^ra_m$ with $a_m \in (\mI_R)_{|\theta_m|+1-n}$, and a real $\gamma \in (\mI_QD)_2$. Write $G(\vec{a})$ in the form of a list as in~\eqref{eq:list}.

Recall the definition of the forms $\bar{b}_m$ from~\eqref{eq:rebbar}.
Let
\[
b_{(0)} = \sum_{m=0}^r a_m\bar{b}_m.
\]
By Lemma~\ref{lm:reR}, the coefficients $a_m$ are real, and therefore $b_{(0)}$ is real.
By Lemma~\ref{lm:init}, the chain $b_{(0)}$ satisfies
\[
\mg(e^{b_{(0)}})\equiv 0=c_{(0)}\cdot 1\pmod{F^{E_0}C},\qquad c_{(0)}=0.
\]
Moreover, $db_{(0)}=0$, $\deg_{C}b_{(0)}=1-n+|\theta_m|+(n-|\theta_m|)=1$.
Since $\nu(b_{(0)}),\nu(\gamma)>0$, Lemmas~\ref{lm:p_zero} and~\ref{lm:czero} guarantee
\[
\Xi_{\phi,2}(b_{(0)},\gamma)
\equiv \bigoplus_m \big(\langle \p^{\beta_0}_{0},\theta_m\rangle_X - (-1)^n\q^{0}_{\emptyset,1}(\zeta)(\theta_m)\big)
\equiv \bigoplus_m 0
\equiv \oplus_m a_m
\pmod{F^{E_0}C}.
\]

Proceed by induction. Suppose we have a real $b_{(l)}\in C$ with $\deg_Cb_{(l)}=1$, $G(b_{(l)})\subset G(\vec{a}),$ and
\begin{gather*}
(db_{(l)})_n =0,\qquad\Xi_{\phi,2}(b_{(l)},\gamma)\equiv\vec{a} \pmod{F^{E_l}C},\\
\mg(e^{b_{(l)}})\equiv c_{(l)}\cdot 1\pmod{F^{E_l}C},\quad c_{(l)}\in (\mI_R)_2.
\end{gather*}
Define the obstruction chains $o_j$ by~\eqref{eq:oj_dfn} and the elements $\u_j^m$ by~\eqref{eq:uclass}.
By Lemma~\ref{lm:oj_closed}, we have $do_j=0$, and by Lemma~\ref{lm:u_even} we have $o_j\in A^{2-\deg\lambda_j}(L).$ To apply Lemma~\ref{lm:ximod}, it is necessary to choose forms $b_j\in A^{1-\deg\lambda_j}(L)$ such that $(-1)^{\deg\lambda_j}db_j=- o_j$
and $\langle b_j,i^*\theta_m\rangle_L=(-1)^{n+1}\u_j^m$,
for all $j\in\{\kappa_l+1,\ldots,\kappa_{l+1}\}$ with $\deg\lambda_j\ne 2$. For each $j$, we also need to verify that $\lambda_jb_j$ is real.

If $\lambda_j$ is not real, Lemma~\ref{lm:reob} and Remark~\ref{rem:ureal} imply that $o_j = 0$ and $\u_j^m=0$, so we choose $b_j = 0.$ 
For real $\lambda_j$, Lemma~\ref{lm:reR} guarantees that $|o_j| \equiv 0, n+1 \pmod 4$. This further splits into different cases as follows.

If $\deg \lambda_j = 2-n,$ Lemma~\ref{lm:td} gives $o_j = 0,$ so we choose $b_j = -\u_j^m\bar{b}_m$.

If $2-n < \deg \lambda_j<~2,$ then $0 < |o_j| < n$.
Then Lemma~\ref{lm:ojex} gives $[o_j] = 0 \in H^*(L;\R)$, and we choose $b_j$ such that $(-1)^{\deg\lambda_j}db_j = - o_j.$
By Lemma~\ref{lm:corrou}, we may choose $b_j$ so that $\langle b_j,i^*\theta_m\rangle_L=(-1)^{n+1}\u_j^m$ for all $m$.

For other possible values of $\deg \lambda_j,$ degree considerations imply $o_j = 0,$ so we choose $b_j =-\u_j^m\bar{b}_m$.

Lemma~\ref{lm:ximod} now guarantees that $b_{(l+1)}:=b_{(l)}+\sum_{\substack{\kappa_l+1\le j\le \kappa_{l+1}\\ \deg\lambda_j\ne 2}}\lambda_jb_j$
satisfies
\[
\Xi_{\phi,2}(b_{(l+1)},\gamma)\equiv \vec{a}\pmod{F^{E_{l+1}}C}
\]
and
\[
\mg(e^{b_{(l+1)}})\equiv c_{(l+1)}\cdot 1\pmod{F^{E_{l+1}}C},\qquad c_{(l+1)}\in (\mI_R)_2.
\]
By construction, $b_{(l+1)}$ is real.
Since $db_j = 0$ when $\deg \lambda_j \in \{2-n,1-n\}$, it follows that $(db_{(l+1)})_n = (db_{(l)})_n=0$.

Thus, the inductive process gives rise to a convergent sequence $\{b_{(l)}\}_{l=0}^\infty$ where $b_{(l)}$ is bounding modulo $F^{E_l}C$.
Taking the limit as $l$ goes to infinity, we obtain
\[
b=\lim_l b_{(l)},\quad\deg_C b=1, \quad \Xi_2(b,\gamma) = \vec{a}, \quad \mg(e^b)= c\cdot 1,\quad c=\lim_lc_{(l)}\in (\mI_R)_2.
\]
\end{proof}

\subsubsection{Injectivity of the classifying map}

Here we formulate an analog of Lemma~\ref{lm:preexact} in the real setting. The proof is verbatim the same.
\begin{lm}\label{lm:repreexact}
Suppose $\Xi_{\phi,2}(\gamma,b)=\Xi_{\phi,2}(\gamma',b')$. Then for all $\xi\in Span([\theta_0]
,\ldots,[\theta_r])\subset H^*(X;\R)$ we have
$\langle\p_{0}^{b,\gamma}-(-1)^n\qg_{\emptyset,1}(\zeta),\xi\rangle_X =\langle\p_{0}^{b',\gamma'}-(-1)^n\q^{\gamma'}_{\emptyset,1}(\zeta),\xi\rangle_X$.
\end{lm}

From now on, we assume that $\Xi_\phi(\gamma,b)=\Xi_\phi(\gamma',b')$.
We will build, by induction on the filtration from Section~\ref{ssec:filtration}, a pair $(\gt,\bt)$ that realizes a gauge equivalence between $(\gamma,b)$ and $(\gamma',b')$. As a first step, we prove an analog of Lemma~\ref{lm:otjex} in the real setting.

Fix a monoid $G$ as in~\eqref{eq:list} such that $G(b),G(b')\subset G$.
Suppose $\gt\in (\mI_Q{\mD})_2$ is closed and real. Let $l\ge 0,$ and suppose we have real $\bt_{(l-1)}\in \mC$ such that $G(\bt_{(l-1)})\subset G$ and $\deg_{\mC}\bt_{(l-1)}=1.$
If $l\ge 1,$ assume in addition that
\begin{gather*}
\mgt(e^{\bt_{(l-1)}})\equiv \ct_{(l-1)}\cdot 1\pmod{F^{E_{l-1}}\mC}, \quad \ct_{(l-1)}\in(\mI_R\mR)_2, \quad d\ct_{(l-1)}=0.
\end{gather*}
For $j\in \{\kappa_{l-1}+1,\ldots,\kappa_l\}$, define the obstruction chains $\ot_j \in A^*(I\times L)$ by
\[
\ot_j:=[\lambda_j](\mgt(e^{\bt_{(l-1)}})).
\]

\begin{lm}\label{lm:reotjex}
For all $j\in \{\kappa_{l-1}+1,\ldots,\kappa_l\}$ such that $1<|\ot_j|<n+1$ and $|\ot_j|\equiv 0,n+1 \pmod 4$, we have
$[\ot_j]=0\in H^*(I\times L,\d(I\times L);\R)$.
\end{lm}
\begin{proof}
The proof is essentially a refinement of the argument give in Lemma~\ref{lm:otjex}, as follows.

By the assumptions on $\bt_{(l-1)}$ and by Lemma~\ref{lm:ut_relative}, we can write
\[
\mt^{\bt_{(l-1)},\gt}_{0}
\equiv \ct_{(l-1)}\cdot 1
+\sum_{\substack{\kappa_{l-1}+1\le j\le \kappa_{l}\\ \deg\lambda_j = 2}}\lambda_j \ct_j\cdot 1
+\sum_{\substack{\kappa_{l-1}+1\le j\le \kappa_{l}\\ \deg\lambda_j\ne 2}}\lambda_j \ot_j \pmod{F^{E_{l}}C}
\]
with $\ct_j=c_j\in \R$.
Set $\chi_l:=\ct_{(l-1)}+\sum_{\substack{\kappa_{l-1}+1\le j\le \kappa_{l}\\ \deg\lambda_j = 2}}\lambda_j \ct_j$.
As in the proof of Lemma~\ref{lm:otjex}, we get
\begin{equation}\label{eq:r_ojt_ex}
0
\equiv
-\chi_l\cdot i_*1
-\sum_{\substack{\kappa_{l-1}+1\le j\le \kappa_{l}\\ \deg\lambda_j\ne 2}}\lambda_ji_*\ot_j
+
d\pt_{0}^{\bt_{(l-1)},\gt}
+
(-1)^{n+1} \qt_{\emptyset,1}^{\gt}(i_*1)
\pmod{F^{E_{l}}C}.
\end{equation}

Fix $j_0\in \{\kappa_{l-1}+1,\ldots,\kappa_l\}$ such that $1<|\ot_{j_0}|<n+1$ and $|\ot_{j_0}|\pmod 4 \in \{0,n+1\}$. Equivalently, $1-n<\deg\lambda_{j_0}<1$ and $\deg\lambda_{j_0}\pmod 4 \in \{2,1-n\}$. In particular, $\deg\lambda_{j_0}\ne 2$.
Take an arbitrary closed $\xit\in A^{n+1-|\ot_{j_0}|}(I\times L)$. We need to verify that it satisfies $\langle \ot_{j_0},\xit\rangle_{I\times L}=0$.

Using the assumption $|\ot_{j_0}|>1$ we have $|\xit|<n$.
Using $|\ot_{j_0}|\equiv 0,n+1\pmod 4$ we have $|\xit|\equiv 0, n+1 \pmod 4$.
In other words, $|\xit|\in \mQ(n)$.
Since $H^*(I\times M)\simeq H^*(M)$ for $M=L,X,$ there exists $\etat\in Span(\theta_0,\ldots,\theta_r)$ such that $i^*[\etat]=[\xit]\in H^{n+1-|\ot_{j_0}|}(I\times L)$.
Pairing equation~\eqref{eq:r_ojt_ex} above with $\xit$ yields
\begin{align*}
\langle \ot_j,\xit\rangle_{I\times L} &
= \langle \ot_j, i^*\etat\rangle_{I\times L}\\
& = \int_I \ll i_*\ot_j,\etat\gg_X\\
& = [\lambda_j]\big(\int_I \ll \chi_l\cdot i_*1 + d \pt^{\bt_{(l-1)},\gt}_{0} + (-1)^{n+1}\qt_{\emptyset,1}^{\gt}(i_*1), \etat\gg_X\big).
\end{align*}
Proceed to compute the right-hand side by the same arguments as in Lemma~\ref{lm:otjex}, with Lemma~\ref{lm:repreexact} used instead of Lemma~\ref{lm:preexact}. We get
\[
\int_I \ll \chi_l\cdot i_*1 - d \pt^{\bt_{(l-1)},\gt}_{0} - (-1)^{n+1}\qt_{\emptyset,1}^{\gt}, \etat\gg_X=0.
\]
We conclude that
\[
\langle \ot_j,\xit\rangle_{I\times L}
=0
\]
for arbitrary closed $\xit$.
Furthermore, Lemma~\ref{lm:ut_relative} gives $j_0^*\ot_j=j_1^*\ot_j=0$.
It follows from Poincar\'e-Lefschetz duality that
\[
[\ot_j]=0\in H^*(I\times L, \d(I\times L);\R).
\]

\end{proof}

We are now ready to prove injectivity of the classifying map. The proof is quite parallel to that of Proposition~\ref{prop:inj}, but as with Proposition~\ref{prop:spin_b_exist}, we need to verify in every step of the process that the resulting chains are real. In effect, this is the same as the proof of~\cite[Proposition 4.15]{ST2} but with Lemma~\ref{lm:reotjex} used instead of the much simpler~\cite[Lemma 3.18]{ST2}. We give the full argument here for completeness.

\begin{prop}\label{prop:spin_unique}
Suppose $n \not \equiv 1 \pmod 4,$ and $i^*:H^j:(X;\R)\to H^j(L;\R)$ is surjective for all $j<n$ with $j \equiv 0, 1+n \pmod 4$.
Let $(\gamma,b)$ be a bounding pair with respect to $J$ and let $(\gamma',b')$ be a bounding pair with respect to $J'$, both real, such that $\Xi_\phi([\gamma,b])=\Xi_\phi([\gamma',b']).$
Then $(\gamma,b)\sim(\gamma',b')$.
\end{prop}

\begin{proof}
We construct a pseudoisotopy $(\mgt,\bt)$ from $(\mg,b)$ to $(\m^{\gamma'},b').$
Let $\{J_t\}_{t \in [0,1]}$ be a path from $J$ to $J'$ in $\J_\phi$.
Choose a real $\xi\in (\mI_Q D)_1$ such that
\[
\gamma'-\gamma=d\xi
\]
and define
\[
\gt:=\gamma+t(\gamma'-\gamma)+dt\wedge\xi\in \mD.
\]
Then $\gt$ is real, $\deg_\mD \gt = 2$, and
\begin{gather*}
d\gt=dt\wedge\d_t(\gamma+t(\gamma'-\gamma))-dt\wedge d\xi=0,\\
j_0^*\gt=\gamma,\quad j_1^*\gt=\gamma',\quad i^*\gt=0.
\end{gather*}

We now move to constructing $\bt$.
Write $G(b,b')$ in the form of a list as in~\eqref{eq:list}.
Since $\int_Lb'=\int_Lb$ and $b,b',$ are real, there exists a real $\eta\in \mI_RC$ such that $G(\eta)\subset G(b',b)$, $|\eta|=n-1$, and $d\eta=(b')_n-(b)_n.$
Write
\[
\bt_{(-1)}:=b+t(b'-b)+dt\wedge\eta\in\mC.
\]
Then $\bt_{(-1)}$ is real, $\deg_{\mC}\bt_{(-1)}=1$, $G(\bt_{(-1)})\subset G(b',b)$, and
\begin{gather*}
(d\bt_{(-1)})_{n+1}=dt\wedge(b'-b)_n-dt\wedge d\eta=0,\\
j_0^*\bt_{(-1)}=b,\qquad j_1^*\bt_{(-1)}=b'.
\end{gather*}
Let $l\ge 0$. Assume by induction that we have constructed a real $\bt_{(l-1)}\in \mC$ with $\deg_{\mC}\bt=1$ and $G(\bt_{(l-1)})\subset G(b,b')$, such that
\[
(d\bt_{(l-1)})_{n+1}=0,\qquad j_0^*\bt_{(l-1)}=b,\qquad j_1^*\bt_{(l-1)}=b',
\]
and if $l\ge 1$, then
\[
\mgt(e^{\bt_{(l-1)}})\equiv \ct_{(l-1)} \cdot 1\pmod{F^{E_{l-1}}\mC}, \qquad \ct_{(l-1)}\in(\mI_R\mR)_2, \qquad d\ct_{(l-1)} = 0.
\]
Define the obstruction chains $\ot_j$ by ~\eqref{eq:ojt_dfn}.
By Lemma~\ref{lm:ut_even} we have $\ot_j\in A^{2-\deg\lambda_j}(I\times L),$ by Lemma~\ref{lm:ojt_closed} we have $d\ot_j=0$, and by Lemma~\ref{lm:ut_relative} we have $\ot_j|_{\d(I\times L)}=0$ whenever $\deg \lambda_j \ne 2$.
To apply Lemma~\ref{lm:ut_exact}, it is necessary to choose forms $\bt_j\in A^{1-\deg\lambda_j}(I\times L)$ such that $(-1)^{\deg\lambda_j}d\bt_j=-\ot_j+\ct_j\, dt, \ct_j \in \R,$ for $j\in\{\kappa_{l-1}+1,\ldots,\kappa_{l}\}$ such that $\deg\lambda_j\ne 2.$

If $\lambda_j$ is not real, Lemma~\ref{lm:reobt} implies that $\ot_j = 0$, so we choose $\bt_j = 0.$

If $\lambda_j$ is real, then Lemma~\ref{lm:reR} implies $|\ot_j| \equiv 0$ or $1+n \pmod 4$. This further splits into different cases as follows.

If $\deg \lambda_j = 1-n,$ Lemma~\ref{lm:deg_ut} implies that $\ot_j = 0,$ so again we choose $\bt_j = 0.$

If $\deg \lambda_j = 1,$ Lemma~\ref{lm:ob1} gives $\bt_j$ such that $-d \bt_j = -\ot_j+\ct_j \,dt,\,\ct_j \in \R,$ and $\bt_j|_{\d(I\times L)} = 0$.

If $1-n < \deg \lambda_j< 1$, then $1 < |\ot_j| < n+1$.
So, Lemma~\ref{lm:reotjex} implies that $[\ot_j] = 0 \in H^*(I\times L,\d(I\times L);\R).$ Thus, we choose $\bt_j$ such that $(-1)^{\deg\lambda_j}d\bt_j = -\ot_j$ and $\bt_j|_{\d(I\times L)} = 0.$

For other possible values of $\deg \lambda_j,$ degree considerations imply $\ot_j = 0,$ so we choose $\bt_j =0.$

Lemma~\ref{lm:ut_exact} now guarantees that
\[
\bt_{(l)}:=\bt_{(l-1)}+\sum_{\substack{\kappa_{l-1}+1\le j\le\kappa_l \\ \deg\lambda_j\ne 2}}\lambda_j\bt_j
\]
satisfies
\[
\mgt(e^{\bt_{(l)}})\equiv \ct_{(l)}\cdot 1\pmod{F^{E_l}\mC},\qquad \ct_{(l)}\in (\mI_R\mR)_2, \qquad d\ct_{(l)} = 0.
\]
Since $\bt_j = 0$ when $\lambda_j$ is not real and when $\deg \lambda_j = 1-n,$ it follows that $\bt_{(l)}$ is real and satisfies
\[
(d\bt_{(l)})_{n+1}= (d\bt_{(l-1)})_{n+1}=0.
\]
Since $\bt_j|_{\d(I\times L)}=0$ for all $j\in\{\kappa_l+1,\ldots,\kappa_{l+1}\}$ such that $\deg\lambda_j\ne 2$, we have
\[
j_0^*\bt_{(l)}=j_0^*\bt_{(l-1)}=b,\qquad j_1^*\bt_{(l)}=j_1^*\bt_{(l-1)} = b'.
\]
Taking the limit $\bt=\lim_l\bt_{(l)},$ we obtain a real bounding chain for $\mgt$ that satisfies $j_0^*\bt=b$ and $j_1^*\bt=b'$. So, $(\mgt,\bt)$ is a pseudoisotopy from $(\mg,b)$ to $(\m^{\gamma'},b').$
\end{proof}

\subsubsection{Proof of the theorem}
We now combine the above propositions into the main result of the section.

\begin{proof}[Proof of Theorem~\ref{thm:B}]
The cohomological conditions are equivalent to $i^*:H^j(X;\R) \to H^j(L;\R)$ for $j \equiv 0,3,n,n+1 \pmod 4$.
Thus,
surjectivity follows from Proposition~\ref{prop:spin_b_exist}, and injectivity from Proposition~\ref{prop:spin_unique}.
\end{proof}

\bibliographystyle{../../amsabbrvcnobysame}
\bibliography{../../bibliography_exp}

\end{document}